\documentclass[11pt,reqno]{amsart}

\usepackage[english]{babel}
\usepackage{enumitem}
\usepackage[letterpaper,top=2cm,bottom=2cm,left=3cm,right=3cm,marginparwidth=1.75cm]{geometry}

\usepackage{amsmath}
\usepackage{amssymb}
\usepackage{graphicx}
\usepackage{tikz}
\usepackage{comment}

\usepackage[colorlinks=true, allcolors=blue]{hyperref}
\theoremstyle{definition}
\newtheorem{theorem}{Theorem}[section]
\newtheorem{corollary}[theorem]{Corollary}
\newtheorem{lemma}[theorem]{Lemma}
\newtheorem{proposition}[theorem]{Proposition}
\newtheorem{defin}[theorem]{Definition}

\title{A Rademacher exact type formula for $\text{pod}_2(n)$}
\author{Kilian Rausch}
\address{Department of Mathematics and Computer Science, Division of Mathematics, University of Cologne, Weyertal 86-90, 50931 Cologne, Germany}
\email{krausch1@uni-koeln.de}
\begin{document}


\begin{abstract}
In this paper, we calculate an exact formula for the number of partitions of a natural number $n$, where the largest part is even and no odd parts appears more than two times. The generating functions of the number of these partitions is a mixed mock modular form of weight 0. In order to obtain the formula we apply an extended version of the circle method, during which we need to bound Kloosterman sums and similar exponential sums as well as Mordell-type integrals.
\end{abstract}

\maketitle
\section{Introduction and statement of results}
For  $n\in \mathbb{N}_0,$ a \textit{partition} $\lambda=(\lambda_1,...,\lambda_s)$ of $n$ is a weakly decreasing sequence of natural numbers  with $\sum_{j=1}^s \lambda_j=n$. In this paper, we are interested in a subset of all partitions, namely, the partitions where the largest number $\lambda_1$ is even and any odd number  appears only up to two times. Denote by $\text{pod}_2(n)$ the number of partitions of $n$ with these properties. By convention, we set $\text{pod}_2(0):=1$ and find
\begin{align*}
    \text{pod}_2(2)= 1, \hspace{0.5cm }  \text{pod}_2(3)=  1,\hspace{0.5cm }  \text{pod}_2(4)=3,  \hspace{0.5cm } \text{pod}_2(5)= 2  \hspace{0.5cm} \text{and} \hspace{0.5cm} \text{pod}_2(6)= 5.
\end{align*}
Banerjee and Bringmann \cite{banbring} showed that the generating function of the number of these partitions  
\begin{align*}
    \text{POD}_2(q):= \sum_{n=0}^\infty \text{pod}_2(n)q^n,
\end{align*}
is closely related to Ramanujan's \textit{third order mock theta function}
\cite[(2.1.7)]{RamanujanV}
\begin{align*}
    \rho(q):= \sum_{n =0}^\infty \frac{q^{2n(n+1)}}{\prod_{j=0}^n\left (1+q^{2j+1}+q^{4j+2}\right)}
\end{align*}
via
\begin{align*}\label{PODrho}
    \text{POD}_2(q) = \frac{\left(q^3;q^3\right)_\infty}{\left(q^6;q^6\right)_\infty(q;q)_\infty} \rho(q). \tag{1.1}
\end{align*}
Here and throughout for $a \in \mathbb{C}$ and $n \in \mathbb{N}_0 \cup \{\infty\}$ we denote the $q$\textit{-Pochhammer symbol}  by $(a;q)_n:=\prod_{j=0}^{n-1} \left(1-aq^j \right).$

 A natural question arising in the context of generating functions is finding an exact type formula for the Fourier coefficients of a given series. A famous example for such a formula is the exact type formula for $p(n)$, the number of partitions of a given natural number $n,$ by Rademacher \cite{Rademacherkloosterman}. He employed the circle method, originally developed by Hardy and Ramanujan, to obtain this result. As it  is stated in terms of \textit{Kloosterman sums}, $A_k(n),$  we define 
\begin{align*}
    A_k(n):= \sum_{\substack{h \text{ (mod } k)\\ \text{gcd}(h,k)=1}} \omega_{h,k} e^{-\frac{2\pi i nh}{k}}.
\end{align*}
For the definition of the multiplier $\omega_{h,k}$ see (\ref{omega}). With this, the exact type formula for $p(n)$ is given by
\begin{align*}
    p(n)=\frac{2 \pi}{\left(24n-1\right)^\frac{3}{4}} \sum_{k=1}^\infty \frac{A_k(n)}{k} I_{\frac{3}{2}} \left( \frac{\pi \sqrt{24n-1}}{6k} \right). \tag{1.2}
\end{align*}
Here, $I_k$ is a modified Bessel function of order $k$ \cite{Temme2015}. 

Building on this, Bridges and Bringmann \cite[Theorem 1.1]{bridges2023rademachertype} used the circle method to derive 
an exact type formula for the number of partitions without sequence\footnote{A partition without sequence is a partition in which no two parts are consecutive numbers.}.   We refrain from stating their formula here, as it contains several types of Kloosterman sums. However, it can be found in \cite[Theorem 1.1]{bridges2023rademachertype}. Bringmann and Manschot then build on this method to derive the first example of an exact type formula for a mixed mock modular form of weight 0 \cite{bringmann2011sheavesp2generalizationrademacher}. 

The main result of this paper is the Rademacher exact type formula for $\text{pod}_2(n),$ which was derived by employing the circle method as presented in the paper by Bridges and Bringmann in \cite{bridges2023rademachertype} and Mauth in \cite{mauth2025exactformula1lowerrun}. In order to state this formula, we need to define for $b \in \mathbb{R},  k\in \mathbb{N} $ and $v\in \mathbb{Z}$ the integrals
\begin{align*}
   \mathbb{I}_{b,k,v}(n):=\int_{-1}^{1} \frac{\sqrt{1-x^2}I_1\left( \frac{4\pi \sqrt{\left(n+\frac{1}{2} \right) \left( 1-x^2\right)}}{bk} \right)}{\tanh \left( \frac{\pi i (6v+2)}{3k} - \frac{2\pi x}{\sqrt{3}bk} \right)} \text{d}x \tag{1.3}
\end{align*}
and
\begin{align*}
    \mathbb{J}_{k,v}(n):= \int_{-1}^{1} \frac{\sqrt{1-x^2}I_1\left( \frac{\pi \sqrt{5\left(n+\frac{1}{2} \right) \left(1-x^2 \right)}}{3k} \right)}{\tanh \left( \frac{\pi i \left(v-\frac{1}{6} \right)}{k} - \frac{\sqrt{5}\pi x}{6\sqrt{3}k} \right)} \text{d}x. \tag{1.4}
\end{align*}
With this, we can now state the main result of this work:
\begin{theorem}\label{mainresult}
    \textit{For} $n \in \mathbb{N}_0$ \textit{a Rademacher exact type formula for} $\text{pod}_2(n)$ \textit{is given by} 
    \begin{align*}\label{finalform}
    \text{pod}_2(n)&= \frac{\pi}{3\sqrt{3\left(n+\frac{1}{2} \right)} } \sum_{\substack{k=1\\\text{gcd}(k,6)=6}}^\infty \frac{1}{k^2} \sum_{v=0}^{\frac{k}{2}-1} (-1)^v  K_k^{[621]}(v,n) \mathbb{I}_{\sqrt{6},k,v}(n)
    \\ 
    &+\frac{\pi }{9\sqrt{3\left(n + \frac{1}{2} \right)}} \sum_{\substack{k=1\\\text{gcd}(k,6)=2}}^\infty \frac{1}{k^2 } \sum_{v=0}^{\frac{k}{2}-1} (-1)^v  K_k^{[221]}(v,n) \mathbb{I}_{3\sqrt{2},k,v}(n)
    \\
    &+\frac{\pi}{3\sqrt{2n+1}}\sum_{\substack{ k=1\\ \text{gcd}(k,6)=2 }}^\infty  \frac{K_k^{[231]}(n)}{k} I_1 \left( \frac{2\pi \sqrt{2n+1  }}{3k} \right)
    \\
    &+\frac{5\pi i}{72 \sqrt{3 \left(n+\frac{1}{2}\right)}} \sum_{\substack{k=1\\ \text{gcd}(k,6)=1}}^{\infty} \frac{1}{k^2} \sum_{v=0}^{k-1}  K_k^{[121]}(v,n) \mathbb{J}_{k,v}(n). \tag{1.5}
\end{align*}
\end{theorem}
Here the $K_k$ are the Kloosterman sums defined in Section 3 (more precisely: at the beginning of Subsection 3.1, 3.2, and 3.4). 

As this paper is rather technical in nature, we want to provide a brief overview of its content. We employ the circle method in order to derive an exact type formula for pod$_2(n)$. For this we rewrite the generating function POD$_2$ of pod$_2(n)$ as a sum of two functions, whose transformation properties are well-known. We then evaluate some Mordell-type integrals as well as give bounds for Kloosterman like sums, that come up in the calculations related to the circle method. With these results, we derive the main result.

This paper is structured in the following way: In Section 2, we supply some preliminaries, we will use in the following sections. Section 3 deals with the Kloosterman sums. Here we will rewrite the multipliers and give bounds for the sums used in Section 4. In the fourth and final section, the circle method is applied to the function POD$_2(q)$ and the main result, as stated above, is derived. 

\section*{Acknowledgements}
I want to express thanks to my dissertation advisor Kathrin Bringmann for suggesting the topic of this paper and for guidance in the  research endeavor. I also want to thank my colleagues Lukas Mauth  for insightful discussions and helpful feedback on the numerical verification of the main result and Caner Nazaroglu for giving valuable feedback to an earlier version of this manuscript. This research presented in this paper was funded by the European Research Council (ERC) under the European Union's Horizon 2020 research and innovation programme (grant agreement No. 101001179). 

\section{Preliminaries}
\subsection{Rewriting POD$_2$}
Ramanujan's third order mock theta function $\omega(q)$ is defined via \cite[(2.1.5)]{RamanujanV}
\begin{align*}
    \omega(q):= \sum_{n=0}^\infty \frac{q^{2n(n+1)}}{(q;q^2)^2_{n+1}}.
\end{align*}
By an identity that also goes back to Ramanujan, $\omega$ is related to $\rho$ by \cite[(2.3.6) and (2.3.11)]{RamanujanV}
\begin{align*}
    2\rho(q)+\omega(q)=  3 \frac{\left(q^6;q^6\right)^2_\infty}{\left(q^3;q^6 \right)_\infty^2 \left(q^2;q^2\right)_\infty}.
\end{align*}
 Rearranging  and simplifying gives
 \begin{align*}
     \rho(q)&= -\frac{1}{2} \omega(q) + \frac{3}{2} \frac{\left(q^6;q^6\right)^2_\infty}{\left(q^3;q^6 \right)_\infty^2\left(q^2;q^2\right)_\infty} 
     \\&=      -\frac{1}{2} \omega(q) + \frac{3}{2}  \frac{\left(q^6;q^6\right)^4_\infty}{\left(q^3;q^3 \right)_\infty^2 \left(q^2;q^2\right)_\infty}.
 \end{align*}
As the modular properties of $\omega(q)$ and $P \left(q^r \right):= \frac{1}{\left(q^r;q^r \right)_\infty}$ for $r \in \mathbb{N}$ are well-understood (see Subsection 2.2), it is advantageous for us to split POD$_2(q)$ into two terms
\begin{align*}
         \text{POD}_2(q)=- \frac{1}{2} \frac{\left(q^3;q^3\right)_\infty}{\left(q^6;q^6\right)_\infty (q;q)_\infty} \omega(q) + \frac{3}{2} \frac{\left(q^6;q^6\right)^3_\infty}{\left(q^3;q^3\right)_\infty  \left(q^2;q^2\right)_\infty (q;q)_\infty}.
\end{align*}
Denote by
\begin{align*}
     \zeta_1(q):= \frac{P\left(q^6\right)P\left(q\right)}{P\left(q^3\right)}= \frac{\left(q^3;q^3\right)_\infty}{\left(q^6;q^6\right)_\infty (q;q)_\infty},
 \end{align*}
 \begin{align*}
     \zeta_2(q):=\frac{P\left(q^3\right)P\left(q^2\right)P(q)}{P\left(q^6\right)^3}=\frac{\left(q^6;q^6\right)^3_\infty}{\left(q^3;q^3\right)_\infty  \left(q^2;q^2 \right)_\infty (q;q)_\infty}.
 \end{align*}
 With this notation, we get 
 \begin{align*}
     \text{POD}_2(q)= -\frac{1}{2} \zeta_1(q)\omega(q) + \frac{3}{2} \zeta_2(q). \tag{2.1}
 \end{align*}
 Next, we study the modular transformation properties of $\omega(q)$ and $P(q),$ as well as the ones of $\zeta_1(q)$ and $\zeta_2(q)$ resulting from them. 
\subsection{Modular properties}
Let $h$ and $k$  be two coprime integers. If $k$ is odd, denote by $h'$ an integer such that $hh'\equiv -1$ (mod $k$) and denote by $k'$ an integer such that $hh'+kk'=-1$. In the case that $k$ is even, we modify the modulus slightly: If $6|k$ we assume that the congruence $hh'\equiv -1$ holds (mod $36k$), i.e., $36|k'$ and if gcd$(k,6)=2$ we require it to hold (mod $4k$), i.e., $4|k'.$ For a complex variable $z$ with $\text{Re}(z)>0,$ we set
\begin{align*}
    q:= e^{\frac{2\pi i}{k} \left(h+iz \right)} \text{ and } q_1:= e^{\frac{2\pi i}{k} \left(h'+\frac{i}{z} \right)}.
\end{align*}
With this we have the classical modularity property \cite{Andrews1976Theory}
\begin{align*}\label{Ptranform}
    P(q)=\omega_{h,k}  z^{\frac{1}{2}} e^{\frac{\pi \left(z^{-1}-z \right)}{12k}} P\left(q_1 \right), \tag{2.2}
\end{align*}
where $\omega_{h,k}$ is a $24k$-th root of unity given by
\begin{align*}\label{omega}
    \omega_{h,k}:= \begin{cases}
       \left(\frac{-k}{h} \right) e^{-\pi i \left( \frac{1}{4}(2-hk-h)+\frac{1}{12}\left(k -\frac{1}{k} \right) \left(2h-h'+h^2h'\right)\right)}&  \text{if } h \text{ is odd},\\
        \left( \frac{-h}{k} \right) e^{-\pi i \left(\frac{1}{4}(k-1)+\frac{1}{12}\left(k-\frac{1}{k}\right) \left(2h-h'+h^2h'\right)\right) }&  \text{if } k \text{ is odd,} \tag{2.3}\end{cases}
\end{align*}
and $\left(\frac{\cdot}{\cdot}\right)$ denotes the Kronecker symbol.
From (\ref{Ptranform}) one can derive the modular transform of $P\left(q^r\right)$ for $r\in \mathbb{N}.$ We will refrain from citing the transformation formula here, but will simply state the transformations for $\zeta_1$ and $\zeta_2$ resulting from it. The formula can be found for example in \cite[Section 2]{Bringmann2011AnEO} and \cite[Section 2]{mauth2025exactformula1lowerrun}.
 We start with the transformation of $\zeta_1.$ Here we have to distinguish between the different cases of gcd$(k,6)$.
 \begin{lemma}
     \textit{We have the following transformations for $\zeta_1.$ If} gcd$(k,6)=6$ \textit{we get that }
 \begin{align*} \label{zetatrans6}
     \zeta_1(q)= \frac{\omega_{h,\frac{k}{6}}\omega_{h,k}}{\omega_{h,\frac{k}{3}}} z^\frac{1}{2} e^{\frac{\pi}{3k} \left(z^{-1}-z\right) }\zeta_1\left(q_1\right),\tag{2.4}
 \end{align*}
\textit{for} gcd$(k,6)=2$,
 \begin{align*} \label{zetatrans2}
     \zeta_1(q)= \frac{\omega_{3h,\frac{k}{2}}\omega_{h,k}} {\omega_{3h,k}} z^\frac{1}{2} e^{\frac{\pi}{9kz}-\frac{\pi z}{3k}} \frac{P\left(q_1 \right)P\left(q_1 ^\frac{2}{3}\right)}{P\left(q_1^\frac{1}{3}\right)}\tag{2.5},
 \end{align*}
\textit{for} gcd$(k,6)=3$,
\begin{align*}\label{zetatrans3}
    \zeta_1(q)= \sqrt{2} \frac{\omega_{2h,\frac{k}{3}}\omega_{h,k}}{\omega_{h,\frac{k}{3}}} z^{\frac{1}{2}} e^{-\frac{\pi}{24kz}-\frac{\pi z}{3k}} \frac{P\left(q_1\right)P\left(q_1^\frac{3}{2}\right)}{P\left(q_1^3\right)}.\tag{2.6}
\end{align*}
\textit{and for} gcd$(k,6)=1$,
 \begin{align*}\label{zetatrans1}
     \zeta_1(q)=\sqrt{2}\frac{\omega_{6h,k}\omega_{h,k}}{\omega_{3h,k}}z^\frac{1}{2} e^{ \frac{5\pi}{72kz}-\frac{\pi z}{3k}} \frac{P\left(q_1\right)P\left(q_1^\frac{1}{6} \right)}{P\left(q_1^\frac{1}{3}\right)}.\tag{2.7}
 \end{align*}
\textit{For $\zeta_2$ we have the following transformations: If }gcd$(k,6)=6,$ \textit{we have }
 \begin{align*}\label{zeta2trans6}
     \zeta_2(q)=\frac{\omega_{h,\frac{k}{3}}\omega_{h,\frac{k}{2}}\omega_{h,k}}{\omega_{h,\frac{k}{6}}^3} e^{-\frac{\pi}{k} \left(z^{-1} -z\right)} \zeta_2 \left(q_1\right), \tag{2.8}
 \end{align*}
 \textit{for} gcd$(k,6)=2$,
 \begin{align*}\label{zeta2trans2}
     \zeta_2(q)=\frac{1}{3}\frac{\omega_{3h,k}\omega_{h,\frac{k}{2}}\omega_{h,k}}{\omega_{3h,\frac{k}{2}}^3}  e^{\frac{\pi}{9kz}+\frac{\pi z}{k}} \frac{P\left(q_1^\frac{1}{3}\right)P\left(q_1^2\right)P\left(q_1\right)}{P\left(q_1^\frac{2}{3}\right)^3},\tag{2.9}
\end{align*}
 \textit{for} gcd$(k,6)=3,$ 
 \begin{align*}\label{zeta2trans3}
     \zeta_2(q)= \frac{1}{2}\frac{\omega_{h,\frac{k}{3}}\omega_{2h,k}\omega_{h,k}}{\omega_{2h,\frac{k}{3}}^3}  e^{\frac{\pi z}{k}} \frac{P\left(q_1^3\right)P\left(q_1^\frac{1}{2}\right)P\left(q_1\right)}{P\left(q_1^\frac{3}{2}\right)^3},\tag{2.10}
 \end{align*}
\textit{and for} gcd$(k,6)=1$ 
 \begin{align*}\label{zeta2trans1}
     \zeta_2(q)= \frac{1}{6}\frac{\omega_{3h,k}\omega_{2h,k}\omega_{h,k}}{\omega_{6h,k}^3} e^{\frac{ \pi}{9kz} + \frac{ \pi z}{k}} \frac{P\left(q_1^\frac{1}{3}\right)P\left(q_1^\frac{1}{2}\right)P\left(q_1\right)}{P\left(q_1^\frac{1}{6}\right)^3}. \tag{2.11}
 \end{align*}
  \end{lemma}
 \subsection{Modular properties of $\omega$}
 In order to state the corresponding transformation formulas for $\omega,$ we need to define Ramanujan's third order mock theta functions $f$,  namely \cite[(2.1.1)]{RamanujanV}
 \begin{align*}
     f(q):= \sum_{n=0}^\infty \frac{q^{n^2}}{(-q;q)^2_n}.
 \end{align*}
 We also require the two  Mordell-type integrals
 \begin{align*}
     J_{k,v}(z):= \int_{-\infty}^\infty \frac{\exp\left(\frac{-6\pi z x^2}{k}\right)}{\tanh\left(\frac{\pi i  \left(v -\frac{1}{6} \right)}{k}-\frac{2\pi z x}{k}\right)}\text{d}x, \tag{2.12}
 \end{align*}
  \begin{align*}
     I_{k,v}(z):= \int_{-\infty} ^\infty \frac{\exp\left( -\frac{6\pi z x^2}{k} \right)}{\tanh \left( \frac{\pi i  \left(6v+2\right)}{3k} - \frac{ 2\pi z x}{k} \right)} \text{d}x. \tag{2.13}
 \end{align*}
 Using \cite[Theorem 2.4]{Andrewstrans} we have for even $k$ that
 \begin{align*} \label{omegakeven}
     \omega \left(q\right) &= \omega \left(e^{\frac{\pi i }{k/2} \left(h+iz \right)} \right)
     \\&= i(-1)^{\frac{1}{2} (h'+1)}\exp\left(- \frac{3\pi i h'k'}{2} + \frac{3\pi i h'}{2k} - \frac{3\pi i h}{2k}\right) \omega_{h,\frac{k}{2}} z^{-\frac{1}{2}} \exp\left(-\frac{4\pi }{3kz} +\frac{4 \pi z}{3k}  \right) \omega \left(q_1 \right) 
      \\&\hspace{0.5cm}-2 k^{-1} (-1)^{\frac{1}{2} (h'+1)} \exp\left( -\frac{3\pi i  h'k'}{2} -\frac{3\pi i h}{2k} \right) \omega_{h,\frac{k}{2}} z^{\frac{1}{2}} \exp \left(\frac{4\pi z}{3k} \right)\\ &\hspace{1cm} \times\sum_{v \text{ }\left(\text{mod }\frac{k}{2}\right)} (-1)^v \exp\left( -\frac{6\pi i h' \mu^2}{k} + \frac{2\pi i h' \mu }{k} \right) I_{k,v}(z), \tag{2.14}
     \end{align*}
     where $\mu=v+\frac{1}{2}.$ If $k$ is odd, we have by \cite[Theorem 2.3]{Andrewstrans}
\begin{align*}\label{omegakodd}
    \omega \left( q\right) &= \omega \left( e^{\frac{\pi i}{k} \left(2h+i2z \right)}\right)
    \\&=  \frac{1}{2 \sqrt{2}} z^{-\frac{1}{2}} (-1)^{\frac{1}{2} (k-1)} \exp\left(\frac{3\pi i h k}{2} - \frac{3\pi i h}{2k} \right) \omega_{2h,k} \exp\left( \frac{\pi}{24kz}+\frac{4\pi z}{3k}\right) f\left( q_1^\frac{1}{2} \right)
    \\& \hspace{0.5cm}+ i(-1)^{\frac{1}{2}(k-1)}\exp\left( \frac{3\pi i hk}{2} -\frac{3\pi i h}{2k} + \frac{4\pi z}{3k} \right) k ^{-1} (2z)^{\frac{1}{2}} \omega_{2h,k} \\
    &\hspace{1cm} \times\sum_{v \text{ } \left(\text{mod } k \right)} \exp\left( \frac{-3\pi i h'v^2}{2k} + \frac{\pi i h'v}{2k} \right)J_{k,v}(z). \tag{2.15}
\end{align*}
\subsection{Farey sequence}
For a natural number $N\in\mathbb{N}$ the Farey sequence of order $N$, $F_N$ is an ordered  set of fractions $\frac{h_i}{k_i}$ with $h_i \in \mathbb{N}_0, k_i \in \mathbb{N}$,  $h_i \leq  k_i \leq N$ and gcd$(h_i,k_i)=1.$ It holds that
\begin{align*}
    \frac{h_i}{k_i} < \frac{h_j}{k_j} \text{ if } i<j.
\end{align*}
For adjacent fractions $\frac{h_1}{k_1}<\frac{h}{k}<\frac{h_2}{k_2}$ in the Farey sequence of order $N,$ then we denote by
\begin{align*}
    \vartheta_{h,k}':=\frac{1}{k(k+k_1)}\text{ and } \vartheta_{h,k}'':= \frac{1}{k(k+k_2)}. \tag{2.16}
\end{align*}
Here and throughout we set $z:= k\left(N^{-2}-i\Phi \right) $ with $-\vartheta_{h,k}' \leq \Phi \leq \vartheta_{h,k}''.$ Furthermore, it is a well known fact \cite{Rademacher1938TheFC} that 
\begin{align*}
    \frac{1}{k+k_1},\frac{1}{k+k_2} \leq \frac{1}{N+1}.
\end{align*}
Another property of adjacent Farey fractions is that \cite{Estermann1929}
\begin{align*}
    hk_1-h_1k=1 \text{ and } hk_2-kh_2=-1.
\end{align*}
From this we immediately get
\begin{align*}\label{k1 mod}
    k_1 \equiv h' \text{ (mod } k) \text{ and } k_2\equiv -h' \text{ (mod } k). \tag{2.17}
\end{align*}
\subsection{Rademacher Kloosterman sums and classical bounds}
When employing the circle method, Kloosterman sums appear and need to be bounded during the course of the calculation. A classical bound for these sums goes back to Salié \cite{Salie1933} and is also stated by Rademacher  in \cite{Rademacher1938TheFC}. Rademacher also stated the bound for the second sum given in the next lemma. However, he does not give a direct proof of this, but referes to Estermann  \cite{Estermann1929}. In \cite{Estermann1929} Estermann uses a trick to derive a bound for incomplete Kloosterman sums that utilizes the bound for the complete Kloosterman sums.  We will also employ his trick at a later point in this paper. 
\begin{lemma}\label{lemmaBoundsKM} \textit{Let $k\in\mathbb{N},$ $n,m,\ell\in\mathbb{Z},$ $n\neq 0, $ and $N+1\leq \ell \leq N+k-1.$  Furthermore, denote by $k_1$ the denominator of the fraction proceeding $\frac{h}{k}$ in the Farey sequence of order $N$ and let $\varepsilon>0.$ Then we have }
\begin{align*}
    H_k(n,m):= \sum_{\substack{0 \leq h <k\\\text{gcd}(h,k)=1\\hw\equiv -1 \text{ } (k)}} e^{-\frac{2\pi i }{k} \left(nh-mw \right)} =O_\varepsilon\left(k^{\frac{2}{3}+\varepsilon} \text{gcd}(|n|,k)^\frac{1}{3} \right), \tag{2.18}
    \\
     \mathbb{H}_{k,l}(n,m):= \sum_{\substack{0 \leq h < k \\\text{gcd}(h,k)=1\\hw\equiv -1 \text{ } (k)\\N\leq k+k_1\leq l}} e^{-\frac{2\pi i }{k} \left(nh-mw \right)} =O_\varepsilon\left(k^{\frac{2}{3}+\varepsilon} \text{gcd}(|n|,k)^\frac{1}{3} \right) \tag{2.19}.
\end{align*}
\textit{In the sums we use the abbreviation $hw\equiv-1 \text{ }(k)$ instead of $hw\equiv -1$ (mod $k$).} 
\end{lemma}
For abbreviation, we write $H_k(n):=H_k(n,0).$ We refer to sums of the second type, where we have an additional condition on $h$ as \textit{incomplete Kloosterman sums}. As some of the Kloosterman sums we will later encounter are somewhat atypical compared to the ones defined above\footnote{The sums we consider often have  different moduli conditions, i.e., if $6|k$ we require $hh'\equiv -1$ (mod 36$k$) and in some cases arguments that are not integral but half integral.} , we need to define another family of Kloosterman like sums.
\begin{defin}\label{MyKMSums}
    \textit{Let $N\in \mathbb{N}$, $k\in \mathbb{N}$ be even, $N+1 \leq \ell \leq N+k-1$, $a\in\{1,2\}$ and $n$ and $m$ half integers. Further denote by $k_1$ the denominator of the fraction proceeding $\frac{h}{k}$  in the Farey sequence of order $N$. Then we define}
    \begin{align*}
        K_{k,a}(n,m):= \sum_{\substack{0\leq h<ak\\ \text{gcd}(h,ak)=1 \\hh'\equiv -1 \text{ }(k^*)  }}e^{\frac{2\pi i}{ak} \left(-nh+mh' \right)}, \tag{2.20}
    \end{align*}
    \begin{align*}
        \mathbb{K}_{k,\ell}(n,m):= \sum_{\substack{0\leq h<k\\ \text{gcd}(h,k)=1 \\hh'\equiv -1 \text{ }(k^*) \\ N<k+k_1\leq \ell }}e^{\frac{2\pi i}{k} \left(-nh+mh' \right)}, \tag{2.21}
    \end{align*}
    \textit{where}
    \begin{align*}
        k^*:= \begin{cases}
            4k &  \text{if gcd}(k,6)=2,\\
            36k & \text{if }6|k.
        \end{cases}
    \end{align*}
\end{defin}
We will show that the bounds similar to the ones in Lemma \ref{lemmaBoundsKM} also apply to the sums defined above. For this, we firstly show that we can rewrite sums where both arguments are elements of $\mathbb{Z}+ \frac{1}{2}$ and $a=1$ as a sum with the same $k,$ where $a=2$ and the same arguments multiplied by 2, making them integral: 
\begin{lemma}\label{Modulshift}
\textit{For even $k \in \mathbb{N}$ and odd integers $m,n$ it holds that }   
\begin{align*}
    K_{k,2}(n,m)=2 K_{k,1}\left(\frac{n}{2},\frac{m}{2} \right).
\tag{2.22}\end{align*}
\end{lemma}
\begin{proof}
With the definitions above, we have
\begin{align*}
    K_{k,2}(n,m)&=\sum_{\substack{0\leq h<2k\\ \text{gcd}(h,2k)=1 \\hh'\equiv-1 \text{ }(k^*) }}e^{\frac{2\pi i}{2k} \left(-nh+mh' \right)}
    \\&=\sum_{\substack{0\leq h <k \\ \text{gcd}(h,k)=1\\hh'\equiv -1 \text{ }(k^*) }} e^{\frac{2 \pi i}{k} \left( -\frac{n}{2} h+ \frac{m}{2}h' \right)} +\sum_{\substack{k\leq h <2k \\ \text{gcd}(h,k)=1\\hh'\equiv -1  \text{ }(k^*) }} e^{\frac{2 \pi i}{k} \left(- \frac{n}{2}h+\frac{m}{2}h' \right)}
    \\&= K_{k,1}\left( \frac{n}{2}, \frac{m}{2} \right) + \sum_{\substack{0\leq h <k \\ \text{gcd}(h,k)=1\\(h+k)(h+k)'\equiv -1 \text{ }(k^*)}} e^{\frac{2 \pi i}{k} \left(- \frac{n}{2}(h+k)+\frac{m}{2}(h+k)' \right)},
\end{align*}
as $(h,k)=1$ if and only if $(h,2k)=1,$ as $k$ is even.  
Next, we show that the sum on the right-hand side equals $K_{k,1}\left( \frac{n}{2},\frac{m}{2} \right).$ To see this, we show that for a fixed $h$, we have that $(h+k)'=h'+\ell(h,k) k$ with an odd $\ell(h,k).$ To show this, we give a formula for $\ell(h,k).$
\\If gcd$(k,6)=2$ we have $k^*=4k$ and we choose
\begin{align*}
    \ell(h,k):= -h'[h+k]_4,
\end{align*}
where $[h+k]_4$ is the solution to $(h+k)x\equiv 1$ (mod $4$). A solution for this congruence always exists, as $h+k$ is odd and thus invertible mod 4. This solution then is odd, therefor $\ell(h,k)$ is also odd and we find
\begin{align*}
    (h+k)(h'+\ell(h,k) k)= hh'+ (h' + \ell(h,k)(h+k))k \equiv -1 \text{ (mod }4k),
\end{align*}
by our choice of $\ell(h,k).$ 
If $6|k,$ we have $k^*=36k$ and we set
\begin{align*}
    \ell(h,k):= -h'[h+k]_{36}.
\end{align*}
Using a similar argument as before, we find $\ell(h,k)$ to be well-defined and odd. With these considerations we return to the sum in question and get 
\begin{align*}
    &\sum_{\substack{0\leq h <k \\ \text{gcd}(h,k)=1\\(h+k)(h+k)'\equiv -1 \text{ }(k^*) }} \hspace{-.75cm}e^{\frac{2 \pi i}{k} \left(- \frac{n}{2}(h+k)+\frac{m}{2}(h+k)' \right)}
    = \sum_{\substack{0\leq h <k \\ \text{gcd}(h,k)=1\\(h+k)(h+k)'\equiv -1  \text{ }(k^*)}}\hspace{-.75cm} (-1)^n e^{\frac{2 \pi i}{k} \left(- \frac{n}{2}h+\frac{m}{2}(h+k)' \right)}
    \\&= \sum_{\substack{0\leq h <k \\ \text{gcd}(h,k)=1\\hh'\equiv -1 \text{ }(k^*)  }} (-1)^n e^{\frac{2 \pi i}{k} \left(- \frac{n}{2}h+\frac{m}{2}(h'+\ell(h,k) k) \right)}
     =\sum_{\substack{0\leq h <k \\ \text{gcd}(h,k)=1\\hh'\equiv -1 \text{ }(k^*) }} (-1)^{n+\ell(h,k) m} e^{\frac{2 \pi i}{k} \left(- \frac{n}{2}h+\frac{m}{2}h' \right) }
    \\ &= \sum_{\substack{0\leq h <k \\ \text{gcd}(h,k)=1\\hh'\equiv -1 \text{ }(k^*) }} e^{\frac{2 \pi i}{k} \left(- \frac{n}{2}h+\frac{m}{2}h' \right) }
    =K_{k,1}\left( \frac{n}{2}, \frac{m}{2}\right),
\end{align*}
as both $n$ and $m\ell(h,k)$ are odd. This establishes our claim. 
\end{proof}
Next, we show that we can rewrite our sums, as classical Kloosterman sums, where in some cases, we have to shift the modulus in the classical Kloosterman sum from $k$ to $2k.$
\begin{theorem}\label{KMequals}
\textit{For even $k$ and odd integers $m,n$ we have that}
\begin{align*}\label{ThrmKMequal1}
    2K_{k,1}\left(\frac{n}{2}, \frac{m}{2} \right)= K_{k,2}(n,m)= H_{2k}(n,m), \tag{2.23}
\end{align*}
\begin{align*}\label{ThrmKMequal2}
    K_{k,1}(n,m)=H_k(n,m). \tag{2.24}
\end{align*}
\end{theorem}
\begin{proof}
    The first equality of (\ref{ThrmKMequal1}) is the result of Lemma  \ref{Modulshift}. For the second equality of (\ref{ThrmKMequal1}), we have 
    \begin{align*}
    K_{k,2}\left(n,m \right) &= \sum_{\substack{0\leq h <2k \\ \text{gcd}(h,2k)=1 \\ hh'\equiv -1  \text{ }(k^*)}} e^{\frac{2\pi i}{2k} (-nh+mh')}
    \\&= \sum_{\substack{0\leq h <2k \\ \text{gcd}(h,2k)=1 \\ hw\equiv -1 \text{ }(2k)}} e^{\frac{2\pi i}{2k} (-nh+mw)} =H_{2k}(n,m).
\end{align*}
Note that in the calculations we changed from $hh'\equiv -1$ (mod $k^*$) to $hw\equiv -1$ (mod $2k$). To show that this does not change the sum in question, we need the following considerations: Firstly, any $h'$ that satisfies $hh'\equiv -1 $ (mod $k^*$), also satisfies the congruence mod $2k.$  Conversely, if we have a $w$ such that $hw\equiv -1$ (mod $2k),$ we find that either
\begin{align*}
    hw\equiv 2ku(w,k) -1  \text{ (mod } k^*),
\end{align*}
for $u(w,k)\in \{0,1\}$ if $k^*=4k$ or $u(w,k)\in\{0,...,17\},$ if $k^*=36k.$ 
If $u(w,k)=0$ then there is nothing to do in either case. Otherwise we modify $w$ by subtracting $2ku(w,k)[h]_{k^*}$ from it. This gives
\begin{align*}
    h \left(w-2ku(w,k)[h]_{k^*} \right) = hw -2ku(w,k)h[h]_{k^*} \equiv -1 \text{ (mod }k^*).
\end{align*}
This modification however does not change the sum, as $e^{\frac{2\pi i}{k} \left(2ku(w,k)[h]_{k^*} \right)}=1,$ as the exponent is an integer multiple of $2\pi i.$ An analogous argument gives (\ref{ThrmKMequal2}). 
\end{proof}
With this, we get bounds for our Kloosterman like sums from Definition \ref{MyKMSums}. 
\begin{corollary}\label{halfbounds}
    \textit{Let $N \in \mathbb{N},$  $0<k\leq N$ be an even integer,  $m,n$ odd integers, $a,b \in \mathbb{Z}$ and $N+1\leq \ell \leq N+k-1$. Furthermore, let $k_1$ be the denominator of the fraction proceeding $\frac{h}{k}$ in the Farey sequence of order $N$ and $\varepsilon >0,$ then} 
    \begin{gather*}\label{meinebounds}
       K_{k,2}(a,b)= O_\varepsilon\left(k^{\frac{2}{3}+\varepsilon} \text{gcd}(|a|,k)^\frac{1}{3} \right), \tag{2.25}
   \end{gather*}
   \begin{gather*}\label{meinebounds2}
       K_{k,1}\left(\frac{n}{2},\frac{m}{2} \right)= O_\varepsilon\left(k^{\frac{2}{3}+\varepsilon} \text{gcd}(|n|,k)^\frac{1}{3} \right), \tag{2.26}
     \end{gather*}
    \begin{align*}\label{meineBounds3}
        \mathbb{K}_{k,\ell}\left(\frac{n}{2},\frac{m}{2} \right)=O_\varepsilon\left(k^{\frac{2}{3}+\varepsilon} \text{gcd}(|n|,k)^\frac{1}{3} \right) . \tag{2.27}
    \end{align*}
\end{corollary}
\begin{proof}
   Equation (\ref{meinebounds}) is a direct consequence of (\ref{ThrmKMequal1}) and Lemma \ref{lemmaBoundsKM}. (\ref{meinebounds2}) follows directly from (\ref{meinebounds})  after applying Lemma \ref{Modulshift}. For (\ref{meineBounds3}) we can make use of a trick employed by Estermann in \cite[Section 4, p. 94]{Estermann1929}. As this source is written in german, we outline the basic idea of his method here: We start by defining a function $g(h,k)$ such that
   \begin{align*}
       \sum_{\substack{0\leq h<k\\ \text{gcd}(h,k)=1 \\hh'\equiv -1 \text{ }(k^*) \\ N<k+k_1\leq \ell }} e^{\frac{2\pi i}{k} \left(ah+bh' \right)} = \sum_{\substack{0\leq h<k\\ \text{gcd}(h,k)=1 \\hh'\equiv -1 \text{ }(k^*) }} g(h,k) e^{\frac{2\pi i}{k} \left(ah+bh' \right)},
   \end{align*}
   i.e.,
   \begin{align*}
       g(h,k)= \begin{cases}
           1 & \text{if } k+k_1 \leq \ell,
           \\ 0 & \text{if } k+k_1 > \ell.
       \end{cases}
   \end{align*}
   Next, we set $h_0 := \ell -N$ and find that $0<  h_0 < k.$ With this we define another function  $g_1$ on $\mathbb{Z}$ by
   \begin{align*}
       g_1(x):= \begin{cases}
           1 & \text{if } 0 < x \leq h_0 ,\\
           0 & \text{if } h_0 <x\leq k,
       \end{cases} \text{\hspace{0.5cm} and \hspace{0.5cm} } g_1(x+k)=g_1(x).
   \end{align*}
   By this definition we find that
   \begin{align*}
       g(h,k)=g_1(k+k_1-N)
   \end{align*}
   and by (\ref{k1 mod}) also 
   \begin{align*}
       g(h,k)=g_1(h'-N),
   \end{align*}
   due to the $k$-periodicity of $g_1.$ Next we rewrite $g_1$ as
   \begin{align*}
       g_1(x)= \sum_{r=1}^k a_r e^{\frac{2\pi i rx}{k},}
   \end{align*}
   where  the $a_r$ are given by
   \begin{align*}
       a_r:= \frac{1}{k}\sum_{b=1}^k g_1(b)e^{-\frac{2\pi i r b}{k}}=\frac{1}{k} \sum_{b=1}^{h_0} e^{-\frac{2\pi i r b}{k}} =\begin{cases}
           \frac{h_0}{k} & \text{if } r=k, \\
           \frac{1-e^{-\frac{2\pi i rh_0}{k} }}{k\left( e^{\frac{2\pi ir}{k}}-1 \right)} & \text{if } 1 \leq r\leq k-1.
       \end{cases}
   \end{align*}
   Estermann \cite{Estermann1929} then showed that 
   \begin{align*}
       \sum_{r=1}^k|a_r| < \log(4k).
   \end{align*}
   We can also express $g$ in terms of the $a_r$ by setting $b_r:= a_r e^{-\frac{2\pi i rN}{k}}$
   \begin{align*}
       g(h,k)= \sum_{r=1}^k b_r e^{-\frac{2\pi i r h'}{k}}.
   \end{align*}
   He also showed that the $b_r$ also satisfy
   \begin{align*}
       \sum_{r=1}^k |b_r| < \log(4k).
   \end{align*}
   Combining this, we get
   \begin{align*}
     &\sum_{\substack{0\leq h<k\\ \text{gcd}(h,k)=1 \\hh'\equiv -1 \text{ }(k^*) \\ N<k+k_1\leq \ell }} e^{\frac{2\pi i}{k} \left(ah+bh' \right)} = \sum_{\substack{0\leq h<k\\ \text{gcd}(h,k)=1 \\hh'\equiv -1 \text{ }(k^*) }} g(h,k) e^{\frac{2\pi i}{k} \left(ah+bh' \right)}
   \\&  = \sum_{\substack{0\leq h<k\\ \text{gcd}(h,k)=1 \\hh'\equiv -1 \text{ }(k^*) }} \sum_{r=1}^k b_r   e^{\frac{2\pi i}{k} \left(ah+(b-r)h' \right)} = \sum_{r=1}^k b_r \sum_{\substack{0\leq h<k\\ \text{gcd}(h,k)=1 \\hh'\equiv -1 \text{ }(k^*) }}e^{\frac{2\pi i}{k} \left(ah+(b-r)h' \right)} 
   \end{align*}
   and with this
   \begin{align*}
      \left| \mathbb{K}_{k,\ell}\left(\frac{n}{2},\frac{m}{2}\right)\right| &\leq \sum_{r=1}^k |b_r|\cdot  \bigg|\sum_{\substack{0\leq h<k\\ \text{gcd}(h,k)=1 \\hh'\equiv -1 \text{ }(k^*) }}e^{\frac{2\pi i}{k} \left(\frac{nh}{2}+\frac{(m-r)h'}{2} \right)}  \bigg| 
      \\ &= \sum_{r=1}^k |b_r| \cdot \left| K_{k,1}\left(\frac{n}{2},\frac{m-r}{2} \right)\right|
      \\ &\leq \log(4k) O_{\tilde{\varepsilon}}\left(k^{\frac{2}{3}+\tilde{\varepsilon}} \text{gcd}(|n|,k)^\frac{1}{3} \right)= O_\varepsilon\left(k^{\frac{2}{3}+\varepsilon} \text{gcd}(|n|,k)^\frac{1}{3} \right)
   \end{align*}
   by (\ref{meinebounds2}).
\end{proof}
\subsection{Integral bounds and decompositions}
Next, we calculate upper bounds for some integrals, that will occur when applying the circle method in Section 4. We start with defining
\begin{align*}
    \mathcal{I}_{b,k,v}(z):= z e^{\frac{\pi b}{zk}} I_{k,v}(z)=ze^{\frac{\pi b }{zk}}\int_{-\infty}^{\infty} \frac{\exp\left( -\frac{6\pi z x^2}{k} \right)}{\tanh \left( \frac{\pi i (6v+2)}{3k} - \frac{ 2\pi z x}{k}\right)} \text{d}x, \tag{2.28}
\end{align*}
\begin{align*}
    \mathcal{J}_{b,k,v}(z):= z e^{\frac{\pi b}{zk}} J_{k,v}(z)=ze^{\frac{\pi b }{zk}}\int_{-\infty}^{\infty} \frac{\exp\left( -\frac{6\pi z x^2}{k} \right)}{\tanh \left( \frac{\pi i \left(v- \frac{1}{6} \right)}{k} - \frac{ 2\pi z x}{k}\right)} \text{d}x \tag{2.29}
\end{align*}
Now, following the argument of \cite[Lemma 3.1]{Bringmann2011AnEO} we get the following lemma.
\begin{lemma}\label{boundIntI} \textit{The following upper bounds hold:}
\begin{itemize}[leftmargin=0.25in]
    \item[(1)]{\textit{For $b \leq 0,$ it holds that}
\begin{align*}
    \left| \mathcal{I}_{b,k,v}(z)\right| <\!\!<  \left| \frac{\pi(6 v+2)}{3k}  \right|^{-1}.
\end{align*}}
\item[(2)]{\textit{For $b>0$, we have $\mathcal{I}_{b,k,v}(z)= \mathcal{I}_{b,k,v}'(z)+ \varepsilon_{b,k,v}(z), $ where}
\begin{align*}
    \mathcal{I}_{b,k,v}'(z):= \sqrt{ \frac{b}{6}} \int_{-1}^1 \frac{e^{\frac{b\pi}{zk}\left(1-x^2 \right) }}{\tanh\left(  \frac{\pi i (6v+2)}{3k} - \frac{2\pi x\sqrt{b}}{\sqrt6 k}\right)} \text{dx}
\end{align*}
\textit{and}
\begin{align*}
   | \varepsilon_{b,k,v}(z)| <\!\!< \left| \frac{\pi(6v+2)}{3k} \right|^{-1}.
\end{align*}}
\end{itemize}
\textit{Here any implied constant may depend on $b$.}
\end{lemma}
Instead of a lower bound for $\cosh,$ we need the following lower bound for sinh and $0<b<\pi$:
\begin{align*}
    \left|\sinh\left(a+ib \right) \right| = \left| \sinh(a)\cos(b)+i\cosh(a)\sin(b) \right| \geq |\sin(b)| >\!\!> |b|,
\end{align*}
so 
\begin{align*}
    \left|\sinh\left( \frac{\pi i (6v+2)}{3k} - \frac{2\pi i x}{ak}\right)\right| >\!\!> \left| \frac{\pi(6v+2)}{3k} \right|.
\end{align*}
Once again following the argumentation in \cite[Proposition 3.2]{Bringmann2011AnEO} we get the following proposition:
\begin{proposition}\label{propIntI}
\textit{We have for $b>0$ and $n \in \mathbb{N}$ that}
\begin{align*}
     \int_{-\vartheta_{h,k}'}^{\vartheta_{h,k}''} e^{\frac{2\pi nz}{k}} \mathcal{I}'_{b,k,v}(z) \text{d}\Phi = \frac{\pi b}{k \sqrt{3n}} \int_{-1}^1  \frac{\sqrt{1-x^2} I_1\left( \frac{2\pi }{k} \sqrt{ 2bn \left(1-x^2 \right) } \right)}{\tanh\left( \frac{\pi i (6v+2)}{3k} - \frac{2\pi x \sqrt{b}}{
     \sqrt{6}k}\right)} \text{d}x + \varepsilon'_{b,k,v}
\end{align*}
\textit{with }
\begin{align*}|\varepsilon_{b,k,v}'| <\!\!< \left| \frac{1}{N\pi(6v+2)} \right|.
\end{align*}
\end{proposition}
Similar results also hold for the integral involving $J_{k,v}(z).$ Analogous to Lemma \ref{boundIntI} and \ref{propIntI} we have:
\begin{lemma} \label{boundIntJ} \textit{The following upper bounds hold:}
\begin{itemize}[leftmargin=0.25in]
    \item[(1)]{\textit{For $b \leq 0,$ it holds that}
\begin{align*}
    \left| \mathcal{J}_{b,k,v}(z)\right| <\!\!<  \left| \frac{\pi \left(v-\frac{1}{6} \right)}{k}  \right|^{-1}.
\end{align*}} 
\item[(2)]{\textit{For $b>0$, we have $\mathcal{J}_{b,k,v}(z)= \mathcal{J}_{b,k,v}'(z)+ \varepsilon_{b,k,v}(z), $ where}
\begin{align*}
    \mathcal{J}_{b,k,v}'(z):= \sqrt{ \frac{b}{6}} \int_{-1}^1 \frac{e^{\frac{b\pi}{zk}\left(1-x^2 \right) }}{\tanh\left(  \frac{\pi i \left(v-\frac{1}{6}\right)}{k} - \frac{2\pi x\sqrt{b}}{\sqrt6 k}\right)} \text{d}x
\end{align*}
\textit{and}
\begin{align*}
   | \varepsilon_{b,k,v}(z)| <\!\!< \left| \frac{\pi \left(v-\frac{1}{6} \right)}{k} \right|^{-1}.
\end{align*}}
\end{itemize}
\textit{Here any implied constant may depend on $b.$}
\end{lemma}
\begin{proposition}
\textit{We have for $b>0$ and $n \in \mathbb{N}$ that}
\begin{align*}
     \int_{-\vartheta_{h,k}'}^{\vartheta_{h,k}''} e^{\frac{2\pi nz}{k}} \mathcal{J}'_{b,k,v}(z) \text{d}\Phi = \frac{\pi b}{k \sqrt{3n}} \int_{-1}^1  \frac{\sqrt{1-x^2} I_1\left( \frac{2\pi }{k} \sqrt{ 2bn \left(1-x^2 \right) } \right)}{\tanh\left( \frac{\pi i \left(v-\frac{1}{6} \right)}{k} - \frac{2\pi x \sqrt{b}}{
     \sqrt{6}k}\right)} \text{d}x + \varepsilon'_{b,k,v}
\end{align*}
\textit{with}
\begin{align*}
|\varepsilon_{b,k,v}'| <\!\!< \left| \frac{1}{N\pi\left( v-\frac{1}{6}\right)} \right|.
\end{align*}
\end{proposition}
We will also need another proposition stated in the paper of Bringmann and Mahlburg \cite[Proposition 3.3]{Bringmann2011AnEO} (though there is a slight typo in their version, where they send $n \to \infty$ instead of $N$). The proof of this statement goes back to Rademacher and Zuckerman, see \cite{Rademacher1943OnTE} and \cite{Rademacher1938OnTF}.
\begin{proposition}\label{EZInt}
\textit{We have for $r>0$ and as $N \to \infty$}
\begin{align*}
     \int_{-\vartheta_{h,k}'}^{\vartheta_{h,k}''} e^{\frac{2\pi }{k} \left(nz+\frac{r}{z}\right)}  \text{d}\Phi = \frac{2\pi \sqrt{\frac{r}{n}}}{k} I_1 \left( \frac{4\pi \sqrt{nr} }{k} \right) + O\left(  \frac{1}{Nk}\right),
\end{align*}
\textit{where $I_1$ is the modified Bessel function defined by}
\begin{align*}
  \sigma^{-\frac{1}{2}}  I_1(2\sqrt{\sigma}):= \frac{1}{2\pi i} \int_{\gamma -i\infty}^{\gamma +i \infty} t^{-2}e^{\sigma t^{-1} +t} \text{d}t.
\end{align*}
\end{proposition}

\section{Kloosterman sums and bounds}
In this section, we closely examine the Kloosterman sums that occur when we employ the circle method in the following section. For this, we first have to evaluate the multipliers. Once again, we distinguish between the different cases of gcd$(k,6)$:

\subsection{The case gcd$(k,6)=6$} Recall that in this case, we assume that the congruence  $hh'\equiv -1 $ holds (mod $36k$), i.e., that $36|k'$. We start by defining the following three Kloosterman sums\footnote{Some clarification on the notation: All Kloosterman sums here and throughout are labeled by a three digit number. The first two numbers denote the integral in the circle method to which the sum belongs to. To evaluate these integrals, they will in turn be split into three parts. The third number corresponds to the term in the decomposition of the integral the Kloosterman sums belong to. As the second and third term in the decomposition give rise to Kloosterman sums that can be bound in the same way, the third kind is omitted here. For more details see Section 4.}
\begin{gather*}
      K_k^{[611]}(n,m):=\sum_{\substack{0 \leq h <k \\ (k,h)=1 \\ hh'\equiv -1 \text{ }(36k)}}  \frac{\omega_{h,\frac{k}{6} }\omega_{h,\frac{k}{2}} \omega_{h,k}}{\omega_{h, \frac{k}{3}}} e^{\frac{\pi i \left(1+h'-3h'k'\right)}{2}+ \frac{3\pi i \left(h'-h\right)}{2k}  } e^{\frac{2\pi i}{k} \left(-nh +m h' \right)},
    \\
     K_k^{[621]}(v,n,m):=\sum_{\substack{0 \leq h <k\\ (h,k)=1\\ hh'\equiv -1 \text{ }(36k) } }  \frac{\omega_{h,\frac{k}{6} }\omega_{h,\frac{k}{2}} \omega_{h,k}}{\omega_{h, \frac{k}{3}}} e^{\frac{\pi i \left(1+h'-3h'k'\right)}{2} -\frac{3\pi i h}{2k}+ \frac{2\pi i h' \left( \mu-3\mu^2 \right)}{k}} e^{\frac{2\pi i}{k} \left(-nh+mh' \right)},  \\
   K_k^{[631]}(n,m):=\sum_{\substack{0 \leq h <k\\ (h,k)=1\\ hh'\equiv -1 \text{ }(36k)}}  \frac{\omega_{h,\frac{k}{3}}\omega_{h,\frac{k}{2}}\omega_{h,k}}{\omega_{h,\frac{k}{6}}^3} e^{\frac{2\pi i }{k} \left(-nh+mh' \right)}  
\end{gather*}
and their three incomplete counterparts for $N+1 \leq \ell \leq N+k-1$
\begin{gather*}
         \mathbb{K}_{k,\ell}^{[612]}(n,m):= \sum_{\substack{0\leq h <k \\ \text{gcd}(h,k)=1\\ hh'\equiv -1 \text{ }(36k)\\ N<k+k_1\leq \ell}}   \frac{\omega_{h,\frac{k}{6} }\omega_{h,\frac{k}{2}} \omega_{h,k}}{\omega_{h, \frac{k}{3}}} e^{\frac{\pi i\left(1+h'-3h'k'\right)}{2}+ \frac{3\pi i \left(h'-h\right)}{2k}  } e^{\frac{2\pi i}{k} \left(-nh +m h' \right)},
    \\
    \mathbb{K}_{k,\ell}^{[622]}(v,n,m):=\sum_{\substack{0\leq h <k \\ \text{gcd}(h,k)=1\\ hh'\equiv -1 \text{ }(36k)\\ N<k+k_1\leq \ell}}  \frac{\omega_{h,\frac{k}{6} }\omega_{h,\frac{k}{2}} \omega_{h,k}}{\omega_{h, \frac{k}{3}}} e^{\frac{\pi i \left(1+h'-3h'k'\right)}{2} -\frac{3\pi i h}{2k} + \frac{2\pi i h' \left(\mu-3\mu^2\right)}{k}} e^{\frac{2\pi i}{k} \left(-nh+mh' \right)},
    \\
    \mathbb{K}_{k,\ell}^{[632]}(n,m):=\sum_{\substack{0\leq h <k \\ \text{gcd}(h,k)=1\\ hh'\equiv -1 \text{ }(36k)\\ N<k+k_1\leq \ell}} \frac{\omega_{h,\frac{k}{3}}\omega_{h,\frac{k}{2}}\omega_{h,k}}{\omega_{h,\frac{k}{6}}^3} e^{\frac{2\pi i }{k} \left(-nh+mh' \right)}. 
\end{gather*}
Here and throughout we use the following abbreviation: If the last argument $m$ is equal to $0,$ we simply suppress it and write
\begin{align*}
    K_k^{[xyz]}(n):=K_k^{[xyz]}(n,0) \text{ and }  K_k^{[xyz]}(v,n):=K_k^{[xyz]}(v,n,0)
\end{align*}
\begin{align*}
     \mathbb{K}_{k,\ell}^{[xyz]}(n):=\mathbb{K}^{[xyz]}_{k,\ell}(n,0) \text{ and } \mathbb{K}_{k,\ell}^{[xyz]}(v,n):=\mathbb{K}_{k,\ell}^{[xyz]}(v,n,0).
\end{align*}
Next, we evaluate the multipliers that come up in the Kloosterman sums above:
\begin{lemma} \label{Multsgcd6}\textit{We have}
    \begin{align*}
            \frac{\omega_{h,\frac{k}{6} }\omega_{h,\frac{k}{2}} \omega_{h,k}}{\omega_{h, \frac{k}{3}}}= e^{-\frac{2\pi i}{36k} \left(18k  + h\left(-9-9k-4k^2\right) + h' \left(9-2k^2 \right) \right)},
    \end{align*}
    \begin{align*}
    \frac{\omega_{h,\frac{k}{3}}\omega_{h,\frac{k}{2}}\omega_{h,k}}{\omega_{h,\frac{k}{6}}^3} =e^{- \frac{2\pi i}{18k} \left(h \left(9-2k^2 \right) + h'\left(-9-k^2 \right)  \right)}.
\end{align*}
\end{lemma}
\begin{proof} 
Recall, we assume that $hh'\equiv -1 $ (mod $36k$). As $k$ is even, $h$ has to be odd and we have
\begin{align*}
    \frac{\omega_{h,\frac{k}{6} }\omega_{h,\frac{k}{2}} \omega_{h,k}}{\omega_{h, \frac{k}{3}}}= e^{-2 \pi i A},
\end{align*}
where 
\begin{align*}
    A &= \frac{1}{36k} \left( 18k + h \left(-18-9hh'-9k-2k^2+2hh'k^2 \right) + h' \left(9-2k^2 \right) \right) \\
    &\equiv \frac{1}{36k} \left(18k  + h\left(-9-9k-4k^2\right) + h' \left(9-2k^2 \right) \right)\text{ (mod } 1),
\end{align*}
as $hh'\equiv -1 \text{ (mod } 36k)$. The calculation for the other quotient is analogous.  
\end{proof}
With this, we now show that the Kloosterman sums above can be rewritten to be in the form of the sums in Definition \ref{lemmaBoundsKM}. Furthermore, we show that their arguments are integer valued. Therefore, the bounds from Lemma \ref{lemmaBoundsKM} also apply. 
\begin{lemma}\label{6boundlemma}\textit{For $\varepsilon>0$ it holds that}
    \begin{align*}
        &K_{k}^{[611]}(n,m),K_{k}^{[621]}(v,n,m),K_{k}^{[631]}(n,m),\\ &\hspace{1cm}\mathbb{K}_{k,\ell}^{[612]}(n,m),\mathbb{K}_{k,\ell}^{[622]}(v,n,m),\mathbb{K}_{k,\ell}^{[632]}(n,m) <\!\!<_\varepsilon n^{\frac{1}{3}} k^{\frac{2}{3}+\varepsilon}.
    \end{align*}
\end{lemma}
\begin{proof}
    Using Lemma \ref{Multsgcd6} and the definition of the Kloosterman sums above, we find 
    \begin{gather*}
        K_k^{[611]}(n,m)= (-1)^{\frac{3}{2}}\sum_{\substack{0 \leq h <k \\ (k,h)=1\\ hh'\equiv -1 \text{ }(36k)}}  e^{ \frac{2\pi i}{k} \left( -\frac{h\left(36n+18-9k-4k^2 \right)}{36} + \frac{h'\left( 36m+18+9k+2k^2\right)}{36} \right) } ,
        \\
        K_k^{[621]}(v,n,m)= (-1)^{\frac{3}{2}} \sum_{\substack{0 \leq h <k \\ (k,h)=1\\ hh'\equiv -1 \text{ }(36k)}} e^{\frac{2\pi i }{k} \left( - \frac{h\left(36n+18-9k-4k^2 \right) }{36} + \frac{h'\left( 36m-18+9k+2k^2-108v^2-72v \right) }{36} \right)} ,
        \\
        K_k^{[631]}(n,m)= \sum_{\substack{0 \leq h <k \\ (k,h)=1\\ hh'\equiv -1 \text{ }(36k)}} e^{\frac{2\pi i }{k} \left( -\frac{h\left(18n+9-2k^2 \right)}{18} + \frac{h' \left(18m+9+k^2 \right)}{18} \right)}.
    \end{gather*}    
Next, we show that the factors that $h$ and $h'$ are multiplied with are either both integral or half integral. In the first case, we can apply Lemma \ref{lemmaBoundsKM} directly to get the desired bounds as in \cite[Lemma 3.3]{bridges2023rademachertype}. In the other case, we can apply Corollary \ref{halfbounds}. 
We start with the first sum:
If $k$ is divisible by $12$ we have
\begin{align*}
    \frac{36n+18-9k-4k^2}{36}\equiv \frac{1}{2} \text{ (mod 1)} \text{ and } \frac{36m+18+9k+2k^2}{36} \equiv \frac{1}{2} \text{ (mod 1)}.
\end{align*} 
If $12 \nmid k,$ we find, as $6|k$ and $\frac{k}{2}$ is odd, that
\begin{align*}
    \frac{36n+18-9k-4k^2}{36}\equiv \frac{1-\frac{k}{2}}{2} \equiv 0 \text{ (mod 1)} \text{ and } \frac{ 36m+18+9k+2k^2}{36} \equiv \frac{1+\frac{k}{2} }{2}  \equiv 0 \text{ (mod 1)}.
\end{align*}
In this case, we get by Theorem  \ref{KMEQUAL} and Lemma \ref{lemmaBoundsKM} that
\begin{align*}
    \left|K_k^{[611]}(n,m) \right| &= \left|H_{k} \left(n+ \frac{18-9k-4k^2}{36} , m+ \frac{18+9k+2k^2}{36}\right) \right| 
    \\&<\!\!<_\varepsilon \text{gcd}\left( \left|n+ \frac{18-9k-4k^2}{36} \right|,k \right)^\frac{1}{3} k^{\frac{2}{3}+ \varepsilon} \
    \leq_\varepsilon n^\frac{1}{3} k^{\frac{2}{3}+\varepsilon}.
\end{align*}
To bound the gcd expression, we used
\begin{align*}
    \text{gcd}\left( n+ \frac{18-9k-4k^2}{36},k \right) &\leq \text{gcd}\left( 36n+18-9k-4k^2,k \right) \\&= \text{gcd}(36n+18,k)\leq 36n+18 \leq_\varepsilon n
\end{align*}
In the case $12|k,$ we have by Definition \ref{MyKMSums},  Lemma \ref{Modulshift} and Corollary \ref{halfbounds} (\ref{meinebounds}) that
\begin{align*}
    \left|K_k^{[611]}(n,m) \right| &= \left|K_{k,1} \left(n+ \frac{18-9k-4k^2}{36} , m+ \frac{18+9k+2k^2}{36}\right) \right| 
    \\&= \frac{1}{2}\left|K_{k,2} \left(2n+ \frac{18-9k-4k^2}{18} , 2m+ \frac{18+9k+2k^2}{18}\right) \right| 
    \\&<\!\!<_\varepsilon \text{gcd}\left( \left|2n+ \frac{18-9k-4k^2}{18}\right| ,2k \right)^\frac{1}{3} (2k)^{\frac{2}{3}+ \varepsilon} \
    \leq (2n)^\frac{1}{3} (2k)^{\frac{2}{3}+\varepsilon}.
\end{align*}
Which in both cases is exactly the claim for $K_k^{[611]}(n,m).$ Arguing as above, we get that $\frac{ 36m-18+9k+2k^2-108v^2-72v  }{36}$,  is an element of $ \mathbb{Z} +\frac{1}{2}$ if $12|k$ and is integral, if $12 \nmid k$.   Lastly, as $6|k,$ we have that $18|k^2,$ therefore both $\frac{18n+9-2k^2}{18}$ and $\frac{18m+9+k^2}{18}$ are elements of $\mathbb{Z}+\frac{1}{2}.$ The calculations for the bounds for the corresponding Kloosterman sums is analogous to the one for $K_k^{[611]}$.
The calculations necessary to establish the bounds for the sums $\mathbb{K}_{k,\ell}$ are also analogous. 
\end{proof}
\subsection{The case gcd$(k,6)=2$} Recall that in this case, we need the congruence $hh'\equiv -1$ to hold (mod $4k$). We also assume that $3|h'.$ We once again start by defining three Kloosterman sums
\begin{gather*}
        K_k^{[211]}(n,m):= \sum_{\substack{0 \leq h <k \\ (k,h)=1 \\ hh'\equiv -1 \text{ }(4k)\\ 3|h'}}   \frac{\omega_{3h,\frac{k}{2}}\omega_{h,\frac{k}{2}}\omega_{h,k}}{\omega_{3h,k}} e^{\frac{\pi i \left( 1+h'-3h'k' \right)}{2}+\frac{3\pi i \left(h'-h\right)}{2k} }e^{ \frac{2\pi i}{k} \left(-nh+\frac{mh'}{3} \right)},\\
        K_k^{[221]}(v,n,m):= \sum_{\substack{0 \leq h <k \\ (k,h)=1\\ hh'\equiv -1 \text{ }(4k)\\ 3|h'}}     \frac{\omega_{3h,\frac{k}{2}}\omega_{h,\frac{k}{2}}\omega_{h,k}}{\omega_{3h,k}} e^{\frac{\pi i\left(1+h'-3h'k' \right)}{2}  - \frac{3 \pi i h}{2k} + \frac{2\pi i h' \left(\mu-3\mu^2 \right)}{k}}
    e^{\frac{2\pi i}{k} \left(-nh+ \frac{mh'}{3} \right) },
        \\
        K_k^{[231]}(n,m):= \sum_{\substack{0 \leq h <k \\ (k,h)=1\\ hh'\equiv -1 \text{ }(4k)\\ 3|h'}} \frac{\omega_{3h,k}\omega_{h,\frac{k}{2}}\omega_{h,k}}{\omega_{3h,\frac{k}{2}}^3} e^{\frac{2\pi i }{k} \left( -nh+\frac{mh'}{3}\right)}
\end{gather*}
and their incomplete counterparts for $N+1 \leq \ell \leq N+k-1$
\begin{gather*}
        \mathbb{K}_{k,\ell}^{[212]}(n,m):= \sum_{\substack{0 \leq h <k \\ (k,h)=1\\ hh'\equiv -1 \text{ }(4k)\\ N<k+k_1\leq \ell\\ 3|h'}}  \frac{\omega_{3h,\frac{k}{2}}\omega_{h,\frac{k}{2}}\omega_{h,k}}{\omega_{3h,k}} e^{\frac{\pi i \left( 1+h'-3h'k' \right)}{2}+\frac{3\pi i\left( h'-h\right)}{2k}}e^{ \frac{2\pi i}{k} \left(-nh+\frac{mh'}{3} \right)},\\
        \mathbb{K}_{k,\ell}^{[222]}(v,n,m):= \sum_{\substack{0 \leq h <k \\ (k,h)=1\\ hh'\equiv -1 \text{ }(4k)\\ N<k+k_1\leq \ell\\ 3|h'}}     \frac{\omega_{3h,\frac{k}{2}}\omega_{h,\frac{k}{2}}\omega_{h,k}}{\omega_{3h,k}} e^{\frac{\pi i\left(1+h'-3h'k' \right)}{2} - \frac{3 \pi i h}{2k} + \frac{2\pi i h' \left(\mu-3\mu^2\right)}{k}}
    e^{\frac{2\pi i}{k} \left(-nh+ \frac{mh'}{3} \right) },
        \\
        \mathbb{K}_{k,\ell}^{[232]}(n,m):= \sum_{\substack{0 \leq h <k \\ (k,h)=1\\ hh'\equiv -1 \text{ }(4k) \\ N<k+k_1\leq \ell\\ 3|h'}} \frac{\omega_{3h,k}\omega_{h,\frac{k}{2}}\omega_{h,k}}{\omega_{3h,\frac{k}{2}}^3} e^{\frac{2\pi i }{k} \left( -nh+\frac{mh'}{3}\right)}.
\end{gather*}
As in the previous section, we first rewrite the multipliers from the Kloosterman sums.
\begin{lemma}\label{multigcd2}
    \textit{It holds that}
\begin{align*}
    \frac{\omega_{3h,\frac{k}{2}}\omega_{h,\frac{k}{2}}\omega_{h,k}}{\omega_{3h,k}}=e^{  \frac{2\pi i}{36k} \left( h \left(9+9k \right) + h'\left(-5+2k^2 \right)  \right)},
\end{align*}
\begin{align*}
    \frac{\omega_{3h,k}\omega_{h,\frac{k}{2}}\omega_{h,k}}{\omega_{3h,\frac{k}{2}}^3} =-e^{ -\frac{2\pi i}{18k} \left(h \left( 9+9k\right) + h'\left(1-k^2 \right)  \right)}.
\end{align*}
\end{lemma}
\begin{proof} 
The proof is analogous to the one of Lemma \ref{Multsgcd6}. 
\end{proof}
We once again closely examine the factors of $h$ and $h'$ of the Kloosterman sums defined above, in order to determine in which case we can use the bounds from Lemma \ref{lemmaBoundsKM} and when we have to resort to Corollary \ref{halfbounds}.
\begin{lemma}\label{2boundlemma}
    \textit{For $\varepsilon>0$ it holds that}
    \begin{align*}
        &K_k^{[211]}(n,m),K_k^{[221]}(v,n,m),K_k^{[231]}(n,m),
        \\ &\hspace{1cm}
        \mathbb{K}_{k,\ell}^{[212]}(n,m),
        \mathbb{K}_{k,\ell}^{[222]}(v,n,m),\mathbb{K}_{k,\ell}^{[232]}(n,m) <\!\!<_\varepsilon n^{\frac{1}{3}} k^{\frac{2}{3}+\varepsilon}.
    \end{align*}
\end{lemma}
\begin{proof} Using Lemma \ref{multigcd2} and the definition of the Kloosterman sums at the beginning of this subsection, we get 
    \begin{gather*}
        K_k^{[211]}(n,m)= (-1)^{\frac{1}{2}} \sum_{\substack{0 \leq h <k \\ (k,h)=1\\ hh'\equiv -1 \text{ }(4k)\\ 3|h'}}e^{\frac{2\pi i }{k} \left( -\frac{h\left(36n+18-9k\right)}{36} + \frac{\frac{h'}{3} \left(12m+22+9k+2k^2 \right)}{12}\right)},\\
        K_k^{[221]}(v,n,m)= (-1)^\frac{1}{2} \sum_{\substack{0 \leq h <k \\ (k,h)=1\\ hh'\equiv -1 \text{ }(4k)\\ 3|h'}}e^{\frac{2\pi i }{k} \left( -\frac{h\left(36n+18-9k\right)}{36} + \frac{\frac{h'}{3} \left(12m-14+9k+2k^2-108v^2-72v \right)}{12}\right)},
        \\
        K_k^{[231]}(n,m)= -\sum_{\substack{0 \leq h <k \\ (k,h)=1\\ hh'\equiv -1 \text{ }(4k)\\ 3|h'}}e^{\frac{2\pi i }{k} \left( -\frac{h\left(18n+9+9k\right)}{18} + \frac{\frac{h'}{3} \left( 6m-1+k^2 \right)}{6}\right)}.
\end{gather*}
We once again start with the first sum. If $4|k,$ then the factor of $h$ is an element of $\mathbb{Z}+\frac{1}{2}.$ If $k$ is only divisible by $2$ and not by $4,$ we have
\begin{align*}
    \frac{18-9k}{36}= \frac{1-\frac{k}{2}}{2}\equiv 0 \text{ (mod } 1),
\end{align*}
as $\frac{k}{2}$ is odd. For the factor of $\frac{h'}{3}$ the same properties hold:
If $4|k$ we find
\begin{align*}
    \frac{22+9k+2k^2}{12} \equiv \frac{11+k^2}{6}\equiv \frac{1}{2} \text{ (mod }1),
\end{align*}
as $3\nmid k$ and therefore $k^2\equiv 1 $ (mod $3$). For $4 \nmid k$ we have
\begin{align*}
    \frac{22+9k+2k^2}{12} = \frac{11+\frac{9k}{2}+k^2}{6}\equiv 0 \text{ (mod} 1),
\end{align*}
as $\frac{9k}{2}$ is odd, we find that $11+\frac{9k}{2}$ is even and both $\frac{9k}{2}$ and $11+k^2$ are divisible by three,  as $k$ is not divisible by $3$, so $k^2\equiv 1$ (mod 3). For the second sum, we find, keeping in mind that both $108$ and $72$ are dividable by $36$,  that if $k$ is divisible by 4 that both 
\begin{align*}
    \frac{18-9k}{36} \equiv \frac{1}{2} \text{ (mod }1) \text{ and } \frac{-14+9k+2k^2 }{12}\equiv \frac{-7+k^2}{6}\equiv\frac{1}{2} \text{ (mod } 1),
\end{align*}
as $k^2\equiv 1$ (mod $3$). If $4 \nmid k,$ we find
\begin{align*}
    \frac{18-9k}{36}= \frac{1-\frac{k}{2}}{2}\equiv 
    0 \text{ (mod }1) \text{ and } \frac{-14+9k+2k^2}{12}= \frac{-7+\frac{9k}{2}+k^2}{6}\equiv 0 \text{ (mod }1),
\end{align*}
as $-7+\frac{9k}{2}$ is even and $-7+k^2\equiv 0$ (mod $3$).
For the last sum, we have that $\frac{9+9k}{18}$ is an element of $\mathbb{Z}+ \frac{1}{2}.$ Also note that $\frac{-1+k^2}{6}$ is also an element of $\mathbb{Z}+\frac{1}{2},$ as $k^2 \equiv 4$ (mod $6)$, because $k$ belongs to the residue class of either $2$ or $4$ (mod 6).

Next, in order to get ride of the condition $3|h',$ we employ the transformations $h' \mapsto 3h'$ and $h \mapsto [3]_{4k} h,$ where $[3]_{4k}$ is the inverse of $3$ (mod $4k$).
 Now we proceed as in the proof of Lemma \ref{6boundlemma} in order to get the desired bounds. The bounds for $\mathbb{K}_{k,\ell}$ are calculated analogously.
\end{proof} 
\subsection{The case gcd$(k,6)=3$} In this section, we assume that $2|h'$ and $hh'\equiv -1$ (mod $k$). Once again, we define three Kloosterman sums
\begin{gather*}
    K_k^{[311]}(n,m):=(-1)^{\frac{k+1}{2}}\sum_{\substack{0 \leq h <k\\\text{gcd}(h,k)=1\\ hh'\equiv -1 \text{ }(k)\\ 2|h'}} \frac{\omega_{2h,\frac{k}{3}}\omega_{2h,k}\omega_{h,k}}{\omega_{h,\frac{k}{3}}}  e^{\frac{3\pi i hk}{2} -\frac{3\pi i h}{2k}} e^{\frac{2\pi i}{k} \left(-nh+ \frac{mh'}{2} \right)},
    \\ K_k^{[321]}(v,n,m):=(-1)^{\frac{k+1}{2}}\sum_{\substack{0\leq h<k \\ \text{gcd}(h,k)=1\\ hh'\equiv -1 \text{ }(k)\\ 2|h'}} \frac{\omega_{2h,\frac{k}{3}}\omega_{2h,k}\omega_{h,k}}{\omega_{h,\frac{k}{3}}} e^{\frac{3\pi i kh}{2}  +\frac{\pi i \left(h' \left(v-3v^2\right)-3h \right)}{2k}} e^{\frac{2\pi i}{k} \left(-nh+ \frac{mh'}{2} \right)}, 
    \end{gather*}
    \begin{gather*}
     K_k^{[331]}(n,m):=\sum_{\substack{0 \leq h <k\\\text{gcd}(h,k)=1\\ hh'\equiv -1 \text{ }(k)\\ 2|h'}}\frac{\omega_{h,\frac{k}{3}}\omega_{2h,k}\omega_{h,k}}{\omega_{2h,\frac{k}{3}}^3} e^{\frac{2\pi i }{k} \left(-nh+ \frac{mh'}{2} \right)}
\end{gather*}
together with their incomplete counterparts for $N+1 \leq \ell \leq N+k-1$
\begin{gather*}
        \mathbb{K}^{[312]}_{k,\ell}(n,m):=(-1)^{\frac{k+1}{2}}\sum_{\substack{0 \leq h <k\\ \text{gcd}(h,k)=1\\ hh'\equiv -1 \text{ }(k)\\N<k+k_1\leq \ell\\ 2|h'}} \frac{\omega_{2h,\frac{k}{3}}\omega_{2h,k}\omega_{h,k}}{\omega_{h,\frac{k}{3}}}  e^{\frac{3\pi i hk}{2} -\frac{3\pi i h}{2k}} e^{\frac{2\pi i}{k} \left(-nh+ \frac{mh'}{2} \right)},
        \end{gather*}
        \begin{gather*}
        \mathbb{K}_{k,\ell}^{[322]}(v,n,m):=(-1)^{\frac{k+1}{2}}\sum_{\substack{0\leq h<k \\ \text{gcd}(h,k)=1\\ hh'\equiv -1 \text{ }(k) \\ N<k+k_1\leq \ell\\ 2|h'}} \frac{\omega_{2h,\frac{k}{3}}\omega_{2h,k}\omega_{h,k}}{\omega_{h,\frac{k}{3}}} e^{\frac{3\pi i kh}{2}   +\frac{\pi i\left( h' \left(v -3v^2 \right)-3h \right)}{2k}} e^{\frac{2\pi i}{k} \left(-nh+ \frac{mh'}{2} \right)}, \end{gather*} 
        \begin{gather*}
        \mathbb{K}_{k,\ell}^{[332]}(n,m):=  \sum_{\substack{0 \leq h <k\\ \text{gcd}(h,k)=1\\ hh'\equiv -1 \text{ }(k)\\N<k+k_1\leq \ell\\ 2|h'}}\frac{\omega_{h,\frac{k}{3}}\omega_{2h,k}\omega_{h,k}}{\omega_{2h,\frac{k}{3}}^3} e^{\frac{2\pi i }{k} \left(-nh+ \frac{mh'}{2} \right)}.
\end{gather*}
We once again start by evaluating the multipliers.
\begin{lemma}\label{multigcd3}
    \textit{It holds that}
    \begin{align*}
    \frac{\omega_{2h,\frac{k}{3}}\omega_{2h,k}\omega_{h,k}}{\omega_{h,\frac{k}{3}}} =-e^{- \frac{2\pi i}{36k} \left(-9k+9k^2+ h \left(-9+5k^2 \right) + h'\left(-2k^2 \right)  \right)},
\end{align*}
\begin{align*}
    \frac{\omega_{h,\frac{k}{3}}\omega_{2h,k}\omega_{h,k}}{\omega_{2h,\frac{k}{3}}^3} =-e^{- \frac{2\pi i}{18k} \left(3k^2+ h \left(9+k^2 \right) + h'\left(-k^2 \right)  \right)}.
\end{align*}
\end{lemma}
\begin{proof} As gcd($k,6)=3$ we have that $k$ is odd and thus we find
\begin{align*}
    \frac{\omega_{2h,\frac{k}{3}}\omega_{2h,k}\omega_{h,k}}{\omega_{h,\frac{k}{3}}} =-e^{- 2\pi i A},
\end{align*}
with
\begin{align*}
    A&= \frac{1}{36k} \left( -9k+9k^2 +h \left( -18-9hh'+10k^2+5hh'k^2 \right) + h'\left(-2k^2 \right)  \right)
    \\ &\equiv \frac{1}{36k} \left(-9k+9k^2+ h \left(-9+5k^2 \right) + h'\left(-2k^2 \right)  \right) \text{ (mod } 1).
\end{align*}
To see this, we recall that $hh'\equiv -1$ (mod $k$) and thus only have to show that $\frac{-9+5k^2}{36}$ is an integer. For  this, recall that $3|k,$ i.e., the fraction reduces to $\frac{-1+5\frac{k^2}{9}}{4}.$ Lastly, recall that $k$ is odd, so that $k^2\equiv 1$ (mod $4)$. The calculation for the other quotient is analogous. 
\end{proof} 
With this, we can now once again determine the bounds for the Kloosterman sums. 
\begin{lemma}\label{3boundlemma}
    \textit{For $\varepsilon>0$ it holds that}
    \begin{align*}
        &K_k^{[311]}(n,m),K_k^{[321]}(v,n,m),K_k^{[331]}(n,m),
        \\ &\hspace{1cm}\mathbb{K}_{k,\ell}^{[312]}(n,m),\mathbb{K}_{k,\ell}^{[322]}(v,n,m),\mathbb{K}_{k,\ell}^{[332]}(n,m) <\!\!<_\varepsilon n^{\frac{1}{3}} k^{\frac{2}{3}+\varepsilon}.
    \end{align*}
\end{lemma}
\begin{proof}
    Using Lemma \ref{multigcd3} we get  
\begin{gather*}
    K_k^{[311]}(n,m)=\sum_{\substack{0 \leq h <k\\\text{gcd}(h,k)=1\\ hh'\equiv -1 \text{ }(k)\\ 2|h'}}    e^{\frac{2\pi i}{k} \left(-\frac{h\left(36n+18-22k^2 \right)}{36}+ \frac{\frac{h' }{2}\left(18m+2k^2\right)}{18} \right)},
    \\ K_k^{[321]}(v,n,m)=\sum_{\substack{0\leq h<k \\ \text{gcd}(h,k)=1\\ hh'\equiv -1 \text{ }(k)\\ 2|h'}}  e^{\frac{2\pi i}{k} \left(-\frac{h\left(36n+18-22k^2 \right)}{36}+ \frac{\frac{h'}{2} \left(18m+2k^2-27v^2+9v \right)}{18} \right)},     
    \\ K_k^{[331]}(n,m)= -e^{- \frac{\pi i k}{3}}\sum_{\substack{0 \leq h <k\\\text{gcd}(h,k)=1\\ hh'\equiv -1 \text{ }(k)\\ 2|h'}}e^{\frac{2\pi i }{k} \left(-\frac{h \left(18n+9+k^2\right)}{18} + \frac{ \frac{h'}{2}\left(9m+k^2 \right)}{9} \right)}.
\end{gather*}
As before, we show that the factors with which $h$ and $\frac{h'}{2}$ are multiplied are integers. We start with
\begin{align*}
    \frac{18-22k^2}{36}=\frac{1-11\frac{k^2}{9}}{2}\equiv 0 \text{ (mod } 1),
\end{align*}
as $k$ is divisible by $3$ and odd. Obviously $\frac{2k^2}{18}=\frac{k^2}{9}$ is integral, as $3|k.$ Recall that $k$ is odd, therefore $\frac{9+k^2}{18},$ is integral as well.  Lastly, 
\begin{align*}
    \frac{-27v^2+9v}{18}= \frac{-3v^2+v}{2} \equiv 0 \text{ (mod }1),
\end{align*}
because if $v$ is even, both terms in the numerator are even and if $v$ is odd, their difference is even. To get rid of the condition $2|h'$, we proceed as in the proof of Lemma \ref{2boundlemma}. Next one uses \ref{ThrmKMequal2} and applies Lemma \ref{lemmaBoundsKM}. The calculations for the $\mathbb{K}_{k,\ell}$ are once again analogous. 
\end{proof}
\subsection{The case gcd$(k,6)=1$} Besides the congruence $hh'\equiv -1 $ (mod $k)$, we also assume that $6|h'.$ We start by defining the Kloosterman sums:
\begin{gather*}
    K_k^{[111]}(n,m):= (-1)^{\frac{k+1}{2}}\sum_{\substack{0 \leq h <k\\\text{gcd}(h,k)=1\\ hh'\equiv -1 \text{ }(k)\\6|h'}}\frac{\omega_{6h,k}\omega_{2h,k}\omega_{h,k}}{\omega_{3h,k}} e^{\frac{3\pi i h k}{2} -\frac{3\pi i h}{2k}}e^{\frac{2\pi i }{k} \left(-nh + \frac{mh'}{6}\right)}, 
    \\
     K_k^{[121]}(v,n,m):= (-1)^{\frac{1}{2} (k+1)}\sum_{\substack{0 \leq h <k\\\text{gcd}(h,k)=1\\ hh'\equiv -1 \text{ }(k)\\6|h'}}  \frac{\omega_{6h,k}\omega_{2h,k}\omega_{h,k}}{\omega_{3h,k}}  e^{\frac{3\pi i h k}{2}+\frac{\pi i \left( h' \left( v-3v^2  \right)-3h\right)}{2k}} e^{\frac{2\pi i}{k} \left(-nh + \frac{mh'}{6}  \right)},
    \\
    K_k^{[131]}(n,m):= \sum_{\substack{0 \leq h <k\\\text{gcd}(h,k)=1\\ hh'\equiv -1 \text{ }(k)\\6|h'}} \frac{\omega_{3h,k}\omega_{2h,k}\omega_{h,k}}{\omega_{6h,k}^3}  e^{\frac{2\pi i }{k} \left(-nh + \frac{mh'}{6} \right)}
\end{gather*}
and their incomplete version for $N+1 \leq \ell \leq N+k-1$
\begin{gather*}
    \mathbb{K}_{k,\ell}^{[112]}(n,m):=(-1)^{\frac{k+1}{2}} \sum_{\substack{0 \leq h <k\\\text{gcd}(h,k)=1\\ hh'\equiv -1 \text{ }(k)\\N<k+k_1\leq \ell\\6|h'}}\frac{\omega_{6h,k}\omega_{2h,k}\omega_{h,k}} {\omega_{3h,k}} e^{\frac{3\pi i h k}{2} -\frac{3\pi i h}{2k}}e^{\frac{2\pi i }{k} \left(-nh + \frac{mh'}{6}\right)}, 
    \\
     \mathbb{K}_{k,\ell}^{[122]}(v,n,m):= (-1)^{\frac{1}{2} (k+1)}\sum_{\substack{0 \leq h <k\\\text{gcd}(h,k)=1\\ hh'\equiv -1 \text{ }(k)\\N<k+k_1\leq \ell\\6|h'}} \frac{\omega_{6h,k}\omega_{2h,k}\omega_{h,k}}{\omega_{3h,k}}  e^{\frac{3\pi i h k}{2}+\frac{\pi i \left( h'\left(v-3v^2 \right) -3h\right)}{2k}} e^{\frac{2\pi i}{k} \left(-nh + \frac{mh'}{6}  \right)},
    \\
    \mathbb{K}_{k,\ell}^{[132]}(n,m):= \sum_{\substack{0 \leq h <k\\\text{gcd}(h,k)=1\\ hh'\equiv -1 \text{ }(k)\\N<k+k_1\leq \ell\\6|h'}} \frac{\omega_{3h,k}\omega_{2h,k}\omega_{h,k}}{\omega_{6h,k}^3}  e^{\frac{2\pi i }{k} \left(-nh + \frac{mh'}{6} \right)}.
\end{gather*}
Once again, we start with by evaluating the multipliers.
\begin{lemma}\label{multigcd1} 
    \textit{It holds that}
    \begin{align*}
    \frac{\omega_{6h,k}\omega_{2h,k}\omega_{h,k}}{\omega_{3h,k}}=e^{- \frac{2\pi i}{36k} \left(-9k+9k^2+h \left(-9+9k^2 \right) + h'\left(2-2k^2 \right)  \right)},
\end{align*}
    \begin{align*}
     \frac{\omega_{3h,k}\omega_{2h,k}\omega_{h,k}}{\omega_{6h,k}^3} =e^{\frac{2\pi i}{18k} \left(9h \left( k^2-1\right) + h'\left(k^2-1 \right)  \right)}.
\end{align*}
\end{lemma}
\begin{proof} The proof is analogous to the one of Lemma \ref{multigcd3}.
\end{proof}
With this, we once again obtain the following bounds:
\begin{lemma}\label{1boundlemma}
    \textit{For $\varepsilon>0$ it holds that}
    \begin{align*}
       &K_k^{[111]}(n,m),K_k^{[121]}(v,n,m),K_k^{[131]}(n,m),\\ &\hspace{1cm}\mathbb{K}_{k,\ell}^{[112]}(n,m),\mathbb{K}_{k,\ell}^{[122]}(v,n,m),\mathbb{K}_{k,\ell}^{[132]}(n,m) <\!\!<_\varepsilon n^{\frac{1}{3}} k^{\frac{2}{3}+\varepsilon}.
    \end{align*}
\end{lemma}
\begin{proof}
    Using Lemma \ref{multigcd1} we get
    \begin{gather*}
    K_k^{[111]}(n,m)=-\sum_{\substack{0 \leq h <k\\\text{gcd}(h,k)=1\\ hh'\equiv -1 \text{ }(k)\\ 6|h'}}    e^{\frac{2\pi i}{k} \left(-\frac{h\left(36n+18-18k^2\right)}{36}+ \frac{\frac{h' }{6}\left(6m-2+2k^2\right)}{6} \right)},
    \\ K_k^{[121]}(v,n,m)=-\sum_{\substack{0\leq h<k \\ \text{gcd}(h,k)=1\\ hh'\equiv -1 \text{ }(k)\\ 6|h'}}  e^{\frac{2\pi i}{k} \left(-\frac{h\left(36n+18-18k^2\right)}{36}+ \frac{\frac{h' }{6}\left(6m-2+2k^2-27v^2+9v\right)}{6} \right)},     
    \\ K_k^{[131]}(n,m)= \sum_{\substack{0 \leq h <k\\\text{gcd}(h,k)=1\\ hh'\equiv -1 \text{ }(k)\\ 6|h'}}e^{\frac{2\pi i }{k} \left(-\frac{h \left(2n+1-k^2\right)}{2} + \frac{ \frac{h'}{6}\left(3m-1+k^2 \right)}{3} \right)}.
\end{gather*}
Again, we show that the factors that are multiplied with $h$ and $\frac{h'}{6}$ are integral. 
For all factors of $h$ we find that they are integral as
\begin{align*}
    \frac{18-18k^2}{36}=\frac{1-k^2}{2},
\end{align*}
is integer, as $k$ is odd. for the factors of $\frac{h'}{6}$ we note that  $\frac{-1+k^2}{3},$ is integral as $3\nmid k$ and thus $k^2\equiv 1$ (mod 3). Also $\frac{-27v^2+9v}{6}=\frac{-9v^2+3v}{2}$ is an integer, as if $v$ is even, both terms in the numerator are even and both are odd, when $v$ is odd.  Once again, in order to get rid of the condition $6|h'$, we proceed as in the proof of Lemma \ref{2boundlemma}. To get the desired bounds, one applies Lemma \ref{lemmaBoundsKM}. The calculations for the $\mathbb{K}_{k,\ell}$ are once again analogous. 
\end{proof}
From the proof of Lemma  \ref{1boundlemma} we also get that
\begin{align*}\label{KMEQUAL}
    K_k^{[111]}(n,m)=-K_k^{[131]}(n,m)  \text{ and } \mathbb{K}_{k,\ell}^{[111]}(n,m)=-\mathbb{K}_{k,\ell}^{[131]}(n,m). \tag{3.1}
\end{align*}
We will use this result at a later point, when we show that two principle parts cancel. 

\section{the circle method}
In this section, we will follow Rademacher's approach and employ the circle method \cite{Rademacher1938TheFC}.  We will also closely follow the notation used by Bringmann and Mahlburg in \cite[Section 4.1]{Bringmann2011AnEO}. By Cauchy's integral formula we have 
\begin{align*}
    \text{pod}_2(n)= \frac{1}{2\pi i} \int_{C} \frac{\text{POD}_2(q)}{q^{n+1}} \text{d}q,
\end{align*}
where $C$ is any closed curve inside the unite circle, looping around $0$ with winding number 1. Following the argument of \cite{Bringmann2011AnEO} we find, by setting $C$ to be the circle with radius $r=e^{\frac{-2\pi}{N^2} +2\pi i t }$ and $0\leq t\leq 1$ that
\begin{align*}
    \text{pod}_2(n)=\int_0^1 \text{POD}_2(e^{\frac{2\pi}{N^2}+2\pi i t})\cdot e^{\frac{2\pi n}{N^2}-2\pi i t} \text{d}t.
\end{align*}
With our previous definitions from Section 2, we get 
\begin{align*}
    \text{pod}_2(n) &= \sum_{\substack{0\leq h < k \leq N\\ \text{gcd}(h,k)=1}} e^{-\frac{2\pi i nh}{k}} \int_{-\vartheta'_{h,k}}^{\vartheta_{h,k}''} e^{\frac{2\pi n z}{k}} \text{POD}_2\left(e^{\frac{2\pi i}{k} (h+iz)} \right) \text{d$\Phi$}
    \\ &= \sum_{\substack{0\leq h < k \leq N\\ \text{gcd}(h,k)=1}} e^{-\frac{2\pi i nh}{k}} \int_{-\vartheta'_{h,k}}^{\vartheta_{h,k}''} e^{\frac{2\pi n z}{k}} \left(  g_1\left(e^{\frac{2\pi i}{k} (h+iz)} \right) +g_2\left(e^{\frac{2\pi i}{k} (h+iz)} \right)\right)\text{d$\Phi$}, \tag{4.1}
\end{align*}
where
\begin{align*}
    g_1(q):= -\frac{1}{2} \zeta_1(q)\omega(q)=-\frac{1}{2}\sum_{m\in \mathbb{N}_0} a(n)q^n,
\end{align*} 
\begin{align*}
    g_2(q):=\frac{3}{2} \zeta_2(q)=\frac{3}{2} \sum_{n\in \mathbb{N}_0} b(n)  q^n ,
\end{align*} as well as $z= k\left(N^{-2}-i\Phi \right)$ as in Subsection 2.4.
We will once again differentiate between the various cases of gcd$(k,6)$ as we did in Subsection 2.2 and use  
\begin{align*}
    \text{pod}_2(n)= \sum_6+\sum_2+\sum_3 +\sum_1, \tag{4.2}
\end{align*}
where $\sum_d$ denotes the sum above running over $0\leq h <k\leq N$ with gcd$(h,k)=1$ and gcd$(k,6)=d.$ We can split the integral above into three parts
\begin{align}\label{intsplit}
    \int_{-\vartheta_{h,k}'}^{\vartheta_{h,k}''} =\int_{-\frac{1}{k(k+N)}}^{\frac{1}{k(k+N)}} + \int_{-\frac{1}{k(k+k_1)}}^{-\frac{1}{k(k+N)}} + \int_{\frac{1}{k(k+N)}}^{\frac{1}{k(k+k_2)}}. \tag{4.3}
\end{align} We can split the second integral on the right-hand side of (\ref{intsplit}) even further 
\begin{align*}
    \int_{-\frac{1}{k(k+k_1)}}^{-\frac{1}{k(k+N)}} =\sum_{\ell=k+k_1}^{k+N-1} \int_{-\frac{1}{k\ell}}^{-\frac{1}{k(\ell+1)}}. \tag{4.4}
\end{align*} Using this decomposition, we get
\begin{align*}\label{KMtyp2}
    \sum_{\substack{0\leq h <k\leq N \\ \text{gcd}(h,k)=1 \\
    \text{gcd}(k,6)=d}} \int_{-\frac{1}{k(k+k_1)}}^{-\frac{1}{k(k+N)}}  = \sum_{\substack{0\leq h <k\leq N \\ \text{gcd}(h,k)=1 \\
    \text{gcd}(k,6)=d}}\sum_{\ell=k+k_1}^{k+N-1} \int_{-\frac{1}{k\ell}}^{-\frac{1}{k(\ell+1)}} = \sum_{\substack{1\leq k\leq N\\ \text{gcd}(k,6)=d}} \sum_{\ell=N+1}^{N+k-1} \sum_{\substack{0\leq h<k\\ \text{gcd}(h,k)=1\\ N<k+k_1 \leq \ell}} \int_{-\frac{1}{k\ell}}^{-\frac{1}{k(\ell+1)}}. \tag{4.5}
\end{align*}
Note that the condition $0 \leq h<k$ with gcd$(h,k)=1$ and $N<k+k_1\leq \ell$ is exactly the condition in the Kloosterman sums denoted by $\mathbb{K}$ in Section 3. 
A similar decomposition exists for the third integral on the right-hand side of (\ref{intsplit})
\begin{align*}
    \int_{\frac{1}{k(k+N)}}^{\frac{1}{k(k+k_2)}} = \sum_{\ell=k+k_2}^{N+k-1} \int_{\frac{1}{k(\ell+1)}}^\frac{1}{k\ell}. \tag{4.6}
\end{align*} This decomposition results in
\begin{align*}\label{KMtyp3}
    \sum_{\substack{0\leq h <k\leq N \\ \text{gcd}(h,k)=1 \\
    \text{gcd}(k,6)=d}} \int_{\frac{1}{k(k+N)}}^{\frac{1}{k(k+k_2)}}  = \sum_{\substack{1\leq k\leq N\\ \text{gcd}(k,6)=d}} \sum_{\ell=N+1}^{N+k-1} \sum_{\substack{0\leq h<k\\ \text{gcd}(h,k)=1\\ N<k+k_2 \leq \ell}} \int_{\frac{1}{k(\ell+1)}}^{\frac{1}{k\ell}}. \tag{4.7}
\end{align*}

Note that these third kind of Kloosterman sums, ranging over $0 \leq h<k$ with both gcd$(h,k)=1$ and $N<k+k_2\leq \ell$ are not mentioned above in Section 3. We choose to suppress them there, because the bounds that hold for the Kloosterman sums with $N< k+k_1\leq \ell$ also hold for the ones with $N < k+k_2\leq \ell$. For the regular Kloosterman sums \cite[footnote 8]{Rademacher1938TheFC} as well as the sums defined in Definition \ref{MyKMSums}, we can used a  modified form of the trick Estermann we already utilized in the proof of Corollary \ref{halfbounds}, to see that these bounds hold.

In order to determine the contribution of these integrals to the final result (\ref{finalform}), only a few methods are necessary. Virtually all of these methods have already been demonstrated in \cite[Section 4]{bridges2023rademachertype}. We will refer to them as we need them. 
\subsection{The case gcd$(k,6)=6$}
We have, by using (\ref{zetatrans6}) and (\ref{zeta2trans6}) as well as (\ref{omegakeven}) that
\begin{align*}
    \sum_6= S_{61}+S_{62}+S_{63}, \tag{4.8}
\end{align*}
where
\begin{align*}
    S_{61}:= -\frac{i}{2}\sum_{\substack{0 \leq h < k\leq N\\ \text{gcd}(h,k)=1 \\\text{gcd}(k,6)=6\\hh'\equiv -1\text{ } (36k)}} (-1)^{\frac{1}{2}\left(h'+1\right)}  \frac{\omega_{h,\frac{k}{6} }\omega_{h,\frac{k}{2}} \omega_{h,k}}{\omega_{h, \frac{k}{3}}} e^{-\frac{3\pi i h'k'}{2}+ \frac{3\pi i h'}{2k} -\frac{3\pi i h }{2k} - \frac{2\pi inh}{k}}     
    \\ \times \int_{-\vartheta_{h,k}'}^{\vartheta_{h,k}''} e^{\frac{(2 n+1) \pi z }{k}-\frac{\pi}{zk}} g_1 \left(q_1 \right)\text{d$\Phi$},
    \end{align*}
    \begin{align*}
    S_{62}:=\sum_{\substack{0 \leq h < k\leq N\\ \text{gcd}(h,k)=1 \\\text{gcd}(k,6)=6\\hh'\equiv -1\text{ } (36k)}} \frac{(-1)^{\frac{1}{2}\left(h'+1 \right)}}{k}  \frac{\omega_{h,\frac{k}{6}}\omega_{h,\frac{k}{2}}\omega_{h,k}}{\omega_{h,\frac{k}{3}}} 
e^{-\frac{3\pi i h'k'}{2} -\frac{3\pi i h}{2k} - \frac{2\pi i  nh}{k}}
    \sum_{v \text{ } \left(\text{mod } \frac{k}{2}\right)} (-1)^v e^{-\frac{6\pi i h'\mu^2}{k} + \frac{2\pi i h' \mu}{k}} 
    \\ \times \int_{-\vartheta_{h,k}'}^{\vartheta_{h,k}''}z e^{\frac{(2 n+1) \pi z}{k}+\frac{\pi}{3zk} }\zeta_1\left(q_1\right) I_{k,v}(z)\text{d$\Phi$},
\end{align*}
\begin{align*}
    S_{63}:= \frac{3}{2}\sum_{\substack{0 \leq h < k\leq N\\ \text{gcd}(h,k)=1 \\\text{gcd}(k,6)=6 \\hh'\equiv -1\text{ } (36k)}}  \frac{\omega_{h,\frac{k}{3}}\omega_{h,\frac{k}{2}}\omega_{h,k}}{\omega_{h,\frac{k}{6}}^3} e^{-\frac{2\pi i nh}{k}} \int_{-\vartheta'_{h,k}}^{\vartheta''_{h,k}}  
     e^{-\frac{\pi}{zk} + \frac{(2n+1)z}{k} } g_2 \left(q_1\right)
    \text{d} \Phi.
\end{align*}
\subsubsection{$S_{61}$} We utilize the splitting (\ref{intsplit}). Following the notation in \cite{bridges2023rademachertype} we use the notation $S^{[1]}_{61}, S^{[2]}_{61}$ and $S^{[3]}_{61}$ for the corresponding terms.  We have
\begin{align*}
    S_{61}^{[1]}= -\frac{i}{2}\sum_{\substack{1\leq k\leq N\\6|k}}  \sum_{m \geq 0} a(m) K_k^{[611]}(n,m) \int_{-\frac{1}{k(k+N)}}^{\frac{1}{k(k+N)}} e^{\frac{(2n+1)\pi z}{k} - \frac{(2m+1)\pi }{kz}} \text{d}\Phi.
\end{align*}
Using Lemma \ref{6boundlemma} as well as Re$(z)=\frac{k}{N^2}$ and Re$\left(\frac{1}{z} \right) \geq \frac{k}{2}$, we find that $S_{61}^{[1]} \to 0,$ as $N \to \infty.$  Next, for $S_{61}^{[2]}$ we find that
\begin{align*}
    S_{61}^{[2]}= i \sum_{\substack{1\leq k \leq N\\6|k}} \sum_{m \geq 0} a(m) \sum_{\ell=N+1}^{N+k-1}\mathbb{K}_{k,\ell}^{[612]}(n,m) \int_{-\frac{1}{k\ell}}^{\frac{1}{k(\ell+1)}} e^{\frac{(2n+1)\pi z}{k} - \frac{(2m+1)\pi }{kz}} \text{d}\Phi.
\end{align*}
Again using the bound from Lemma \ref{6boundlemma} as well as the bounds for the real part of both $z$ and $\frac{1}{z},$ we get that $S_{61}^{[2]} \to 0$ as $N \to \infty. $ The calculations for $S_{61}^{[3]}$ are analogous to the ones for $S_{61}^{[2]}$, thus $S_{61}^{[3]} \to 0$ also, as $N \to \infty.$ Thus the contribution of $S_{61}$ to  (\ref{finalform}) is $0.$ 
\subsubsection{$S_{62}$} This integral has both a principal and a non-principal part. The non-principal part vanishes, as $N\to \infty.$ The calculations for the non-principal part are analogous to the ones for integral $\mathcal{S}_{22}$ in \cite{bridges2023rademachertype}. For the principle part, we are left with
\begin{align*}
    \mathbb{S}_{62}:=\sum_{\substack{0 \leq h < k\leq N\\ \text{gcd}(h,k)=1 \\\text{gcd}(k,6)=6\\hh'\equiv -1\text{ } (36k)}} \frac{(-1)^{\frac{1}{2}\left(h'+1 \right)}}{k}  \frac{\omega_{h,\frac{k}{6}}\omega_{h,\frac{k}{2}}\omega_{h,k}}{\omega_{h,\frac{k}{3}}} 
e^{-\frac{3\pi i h'k'}{2} -\frac{3\pi i h}{2k} - \frac{2\pi i  nh}{k}}
    \sum_{v \text{ } \left(\text{mod } \frac{k}{2}\right)} (-1)^v e^{-\frac{6\pi i h'\mu^2}{k} + \frac{2\pi i h' \mu}{k}} 
    \\ \times \int_{-\vartheta_{h,k}'}^{\vartheta_{h,k}''} e^{\frac{(2 n+1) \pi z}{k}} \mathcal{I}_{\frac{1}{3},k,v}(z)\text{d$\Phi$}.
\end{align*}
We once again split  the integral $\mathbb{S}_{62}$ into the  three parts $\mathbb{S}_{62}^{[1]}, \mathbb{S}_{62}^{[2]}$ and $\mathbb{S}_{62}^{[3]}$ using \ref{intsplit}.
Following the methodology in \cite{bridges2023rademachertype} for their integral $\mathcal{S}_{22}^{[1]}$, we get that for $N\to \infty,$ 
\begin{align*}
    \mathbb{S}_{62}^{[1]}= \frac{\pi}{3\sqrt{3\left(n+\frac{1}{2} \right)} } \sum_{\substack{k=1\\\text{gcd}(k,6)=6}}^\infty \frac{1}{k^2} \sum_{v=0}^{\frac{k}{2}-1} (-1)^v  K_k^{[62]}(v,n) \int_{-1}^1 \frac{\sqrt{1-x^2} I_1\left( \frac{4\pi \sqrt{\left(n+\frac{1}{2}\right)\left(1-x^2\right)}}{\sqrt{6}k} \right)}{\tanh\left(\frac{\pi i (6v+2)}{3k} -\frac{2\pi x}{3\sqrt{2} k} \right) } \text{d}x.
\end{align*}
As in \cite{bridges2023rademachertype}, the integrals $\mathbb{S}_{62}^{[2]}$ and $\mathbb{S}_{62}^{[3]}$ do not contribute. 
Therefore,  the whole contribution of the integral $S_{62}$ as $N \to \infty$ is given by $\mathbb{S}_{62}^{[1]}$ which exactly the first term in the formula stated in  Theorem \ref{mainresult}. 
\subsubsection{$S_{63}$} For the integral $S_{63}$ there is no principle part. The contribution vanishes as it did for $S_{61}$ for $N\to \infty.$
\subsection{The case gcd$(k,6)=2$}
If gcd$(k,6)=2$ we find, using (\ref{zetatrans2}) and  (\ref{zeta2trans2}) as well as (\ref{omegakeven}), 
\begin{align*}
    \sum_2&= S_{21}+S_{22}+S_{23}, \tag{4.9}
    \end{align*}
where
\begin{align*}
    S_{21}&:=- \frac{i}{2} \sum_{\substack{0 \leq h < k\leq N\\ \text{gcd}(h,k)=1 \\ \text{gcd}(k,6)=2 \\hh'\equiv -1\text{ } (4k) \\3 |h'}} (-1)^{\frac{1}{2}\left( h'+1 \right)}  \frac{\omega_{3h,\frac{k}{2}}\omega_{h,\frac{k}{2}}\omega_{h,k}}{\omega_{3h,k}} e^{-\frac{3\pi i h'k'}{2}+\frac{3\pi i h'}{2k} - \frac{3\pi i h}{2k} -\frac{2\pi inh}{k}} 
    \\[-1cm]&  \hspace{6cm}    \times \int_{-\vartheta'_{h,k}}^{\vartheta''_{h,k}} e^{-\frac{11\pi}{9zk} +\frac{(2n+1)\pi  z}{k}} \omega\left(q_1\right)\frac{P\left(q_1 \right)P\left(q_1^\frac{2}{3} \right)}{P\left(q_1^\frac{1}{3} \right)} \text{d}\Phi,
    \end{align*}
    \begin{align*}
    S_{22}&:=\sum_{\substack{0 \leq h < k\leq N\\ \text{gcd}(h,k)=1 \\ \text{gcd}(k,6)=2 \\hh'\equiv -1\text{ } (4k) \\ 3 |h'}} \frac{(-1)^{\frac{1}{2}\left(h'+1 \right)}}{k}   \frac{\omega_{3h,\frac{k}{2}}\omega_{h,\frac{k}{2}}\omega_{h,k}}{\omega_{3h,k}} e^{-\frac{3\pi i h'k'}{2} - \frac{3 \pi i h}{2k}-\frac{2\pi inh}{k}} \sum_{v \text{ } \left(\text{mod } \frac{k}{2}\right)} (-1)^v e^{-\frac{6\pi i h'\mu^2}{k} + \frac{2\pi i h' \mu}{k}}
    \\[-1cm]&\hspace{6cm}\times\int_{-\vartheta'_{h,k}}^{\vartheta''_{h,k}} z e^{\frac{\pi }{9zk}+\frac{(2n+1)\pi z}{k}} \frac{P\left(q_1 \right)P\left(q_1^\frac{2}{3} \right)}{P\left(q_1^\frac{1}{3} \right)} I_{k,v}(z) \text{d}\Phi,
\end{align*}
\begin{align*}
    S_{23} := \frac{1}{2}\sum_{\substack{0 \leq h < k\leq N\\ \text{gcd}(h,k)=1 \\ \text{gcd}(k,6)=2\\hh'\equiv -1\text{ } (4k)  \\3|h'}} \frac{\omega_{3h,k}\omega_{h,\frac{k}{2}}\omega_{h,k}}{\omega_{3h,\frac{k}{2}}^3} e^{-\frac{2\pi i nh}{k}} \int_{-\vartheta'_{h,k}}^{\vartheta''_{h,k}}  
      e^{\frac{\pi}{9kz}+\frac{(2n+1)\pi z}{k}} \frac{P\left(q_1^\frac{1}{3}\right)P\left(q_1^2\right)P\left(q_1\right)}{P\left(q_1^\frac{2}{3}\right)^3}\text{d} \Phi.
\end{align*}
\subsubsection{$S_{21}$} As for $S_{11},$ this integral has no contribution as $N \to \infty.$ 
\subsubsection{$S_{22}$} This integral has a principle part. The calculations are analogous to the ones presented for $\mathcal{S}_{22}$ in \cite{bridges2023rademachertype}. The resulting contribution of $S_{22}$ as $N\to \infty$ is exactly the second term in the formula (\ref{finalform}).
\subsubsection{$S_{23}$} The non-principle part of $S_{23}$ vanishes as $N \to \infty.$ This leaves us with
\begin{align*}
   \mathbb{S}_{23}&= \frac{1}{2}\sum_{\substack{0 \leq h < k\leq N\\ \text{gcd}(h,k)=1 \\ \text{gcd}(k,6)=2 \\hh'\equiv -1\text{ } (4k) }} \frac{\omega_{3h,k}\omega_{h,\frac{k}{2}}\omega_{h,k}}{\omega_{3h,\frac{k}{2}}^3} e^{-\frac{2\pi i nh}{k}} \int_{-\vartheta'_{h,k}}^{\vartheta''_{h,k}}  
      e^{\frac{\pi}{9kz}+\frac{(2n+1)\pi z}{k}} \text{d}\Phi
      \\ &=\frac{1}{2}\sum_{\substack{0 \leq h < k\leq N\\ \text{gcd}(h,k)=1 \\ \text{gcd}(k,6)=2\\hh'\equiv -1\text{ } (4k)  }} \frac{\omega_{3h,k}\omega_{h,\frac{k}{2}}\omega_{h,k}}{\omega_{3h,\frac{k}{2}}^3} e^{-\frac{2\pi i nh}{k}} \left(\frac{2\pi}{3\sqrt{2n+1} k} I_1 \left( \frac{2\pi \sqrt{2n+1}}{3k} \right)+ O\left(\frac{1}{kN}\right) \right),
\end{align*}
as $N \to \infty$, by Proposition \ref{EZInt}. Using Lemma \ref{2boundlemma} we get that
\begin{align*}
   \left| \frac{1}{N}\sum_{\substack{  k=1\\  \text{gcd}(k,6)=2 }}^N \frac{1}{k} \sum_{\substack{h=0\\\text{gcd}(h,k)=1\\hh'\equiv -1\text{ } (4k) }}^{k-1}\frac{\omega_{3h,k}\omega_{h,\frac{k}{2}}\omega_{h,k}}{\omega_{3h,\frac{k}{2}}^3} e^{-\frac{2\pi i nh}{k}}  \right| <\!\!< \frac{1}{N} \sum_{k=1}^N  \frac{|K_k^{[231]}(n)|}{k} 
   <\!\!<  \frac{n^{\frac{1}{3}}}{N} \sum_{k=1}^N k^{-\frac{1}{3} +\varepsilon}  \to 0,
\end{align*}
as $N \to \infty.$ So principle part is given by
\begin{align*}
    \frac{\pi}{3\sqrt{2n+1} }\sum_{\substack{0 \leq h < k\leq N\\ \text{gcd}(h,k)=1 \\ \text{gcd}(k,6)=2\\hh'\equiv -1\text{ } (4k)  }} \frac{1}{k}\frac{\omega_{3h,k}\omega_{h,\frac{k}{2}}\omega_{h,k}}{\omega_{3h,\frac{k}{2}}^3} e^{-\frac{2\pi i nh}{k}}  I_1 \left( \frac{2\pi \sqrt{2n+1}}{3k} \right),
\end{align*} which corresponds to the third term in the final result. 
\subsection{The case gcd$(k,6)=3$}
Next, we study the integrals in the case that gcd$(k,6)=3$. Using (\ref{zetatrans3}) and (\ref{zeta2trans3}) as well as  (\ref{omegakodd}), we find that
\begin{align*}
    \sum_3 = S_{31} +S_{32}+S_{33}, \tag{4.10}
\end{align*}
where
\begin{align*}
    S_{13}&:=\frac{1}{4} \sum_{\substack{0 \leq h < k\leq N\\ \text{gcd}(h,k)=1 \\ \text{gcd}(k,6)=3 \\hh'\equiv -1\text{ } (k) \\ 2 |h'}}  (-1)^{\frac{1}{2}(k+1)} \frac{\omega_{2h,\frac{k}{3}}\omega_{2h,k}\omega_{h,k}}{\omega_{h,\frac{k}{3}}}  e^{\frac{3\pi i hk}{2} -\frac{3\pi i h}{2k}-\frac{2\pi i n h}{k}}
    \\[-1cm] &\hspace{4cm}\times \int_{-\vartheta'_{h,k}}^{\vartheta''_{h,k}} e^{\frac{(2n+1)\pi z} {k}} f\left( q_1^{\frac{1}{2}}\right) \frac{P\left(q_1 \right)P\left(q_1^\frac{3}{2} \right)}{P\left(q_1^3 \right)} \text{d}\Phi,
    \end{align*}
    \begin{align*}
    S_{32}&:=i \sum_{\substack{0 \leq h < k\leq N\\ \text{gcd}(h,k)=1 \\ \text{gcd}(k,6)=3 \\hh'\equiv -1\text{ } (k)\\ 2 |h'}} \frac{(-1)^{\frac{1}{2}(k+1)}}{k} \frac{\omega_{2h,\frac{k}{3}}\omega_{2h,k}\omega_{h,k}}{\omega_{h,\frac{k}{3}}} e^{\frac{3\pi i kh}{2} - \frac{3\pi i h}{2k}-\frac{2\pi i nh}{k}}
    \sum_{v \text{ (mod } k)} e^{-\frac{3\pi i h'v^2}{2k} +\frac{\pi i h' v}{2k}}
    \\[-1cm] &\hspace{4cm}\times\int_{-\vartheta'_{h,k}}^{\vartheta''_{h,k}} z e^{-\frac{\pi}{24zk}+\frac{(2n+1)\pi z}{k}} \frac{P\left(q_1 \right)P\left(q_1^\frac{3}{2} \right)}{P\left(q_1^3 \right)} J_{k,v}(z)\text{d}\Phi,
\end{align*}
\begin{align*}
    S_{33} := \frac{3}{4}\sum_{\substack{0 \leq h < k\leq N\\ \text{gcd}(h,k)=1 \\ \text{gcd}(k,6)=3\\hh'\equiv -1\text{ } (k)\\2|h' }}  \frac{\omega_{h,\frac{k}{3}}\omega_{2h,k}\omega_{h,k}}{\omega_{2h,\frac{k}{3}}^3} e^{-\frac{2\pi i nh}{k}} \int_{-\vartheta'_{h,k}}^{\vartheta''_{h,k}}  
      e^{\frac{(2n+1)\pi z}{k}} \frac{P\left(q_1^3\right)P\left(q_1^\frac{1}{2}\right)P\left(q_1\right)}{P\left(q_1^\frac{3}{2}\right)^3}
    \text{d} \Phi.
\end{align*}
\subsubsection{$S_{31}$} Here, we have no principle part. Analogous to the integral $S_{61}$ the contribution vanished as $ N \to \infty.$ 
\subsubsection{$S_{32}$} 
Here, we also have no principal part. The calculations needed to see this, are analogous to the ones for integral $\mathcal{S}_{62}$ in \cite{bridges2023rademachertype}.
\subsubsection{$S_{33}$} Here, once again we have no principle part. As for $S_{61}$ the contribution vanishes as $ N \to \infty.$
\subsection{The case gcd$(k,6)=1$}
Lastly, we find by utilizing (\ref{zetatrans1}) and (\ref{zeta2trans1}) as well as (\ref{omegakodd}) 
\begin{align*}
    \sum_1&= S_{11}+S_{12}+S_{13}, \tag{4.11}
\end{align*}
where
\begin{align*}
    S_{11}&:=\frac{1}{4}   \sum_{\substack{0 \leq h < k\leq N\\ \text{gcd}(h,k)=1 \\ \text{gcd}(k,6)=1 \\hh'\equiv -1\text{ } (k)\\ 6 |h'}} (-1)^{\frac{1}{2}(k+1)} \frac{\omega_{6h,k}\omega_{2h,k}\omega_{h,k}}{\omega_{3h,k}} e^{\frac{3\pi i h k}{2} -\frac{3\pi i h}{2k}-\frac{2\pi i nh}{k}} 
    \\[-1cm]&\hspace{4cm} \times \int_{-\vartheta'_{h,k}}^{\vartheta''_{h,k}} e^{\frac{\pi }{9zk}+\frac{(2n+1)\pi z}{k}} f \left(q_1^{\frac{1}{2}} \right) \frac{P\left(q_1 \right) P\left(q_1^\frac{1}{6} \right) }{P\left(q_1^\frac{1}{3} \right) }\text{d}\Phi,
    \end{align*}
    \begin{align*}
    S_{12}&:=i  \sum_{\substack{0 \leq h < k\leq N\\ \text{gcd}(h,k)=1 \\ \text{gcd}(k,6)=1\\hh'\equiv -1\text{ } (k) \\ 6 |h'}} \frac{(-1)^{\frac{1}{2} (k+1)} }{k}\frac{\omega_{6h,k}\omega_{2h,k}\omega_{h,k}}{\omega_{3h,k}}  e^{\frac{3\pi i h k}{2}-\frac{3\pi i h}{2k}-\frac{2\pi inh}{k}} \sum_{v \text{ (mod } k)} e^{-\frac{3\pi i h'v^2}{2k}+\frac{\pi i h'v}{2k}} \\[-1cm] &\hspace{4cm}\times \int_{-\vartheta'_{h,k}}^{\vartheta''_{h,k}} z e^{\frac{5\pi}{72zk}+\frac{(2n+1)\pi z}{k}} \frac{P\left(q_1 \right) P\left(q_1^\frac{1}{6} \right) }{P\left(q_1^\frac{1}{3} \right)} J_{k,v}(z) \text{d}\Phi,
\end{align*}
\begin{align*}
    S_{13} := \frac{1}{4} \sum_{\substack{0 \leq h < k\leq N\\ \text{gcd}(h,k)=1 \\ \text{gcd}(k,6)=1\\hh'\equiv -1\text{ } (k)\\6|h' }} \frac{\omega_{3h,k}\omega_{2h,k}\omega_{h,k}}{\omega_{6h,k}^3}  e^{-\frac{2\pi i nh}{k}}\int_{-\vartheta'_{h,k}}^{\vartheta''_{h,k}}   e^{\frac{ \pi}{9kz} + \frac{(2n+1) \pi z}{k}}  \frac{P\left(q_1^\frac{1}{3}\right)P\left(q_1^\frac{1}{2}\right)P\left(q_1\right)}{P\left(q_1^\frac{1}{6}\right)^3}  \text{d} \Phi.
\end{align*}
\subsubsection{$S_{11}$}
As it was the case for $S_{23},$ the integral $S_{11}$ has a principle part. It is given by
\begin{align*}
    \mathbb{S}_{11}&= \frac{\pi }{6\sqrt{2n+1}}   \sum_{\substack{0 \leq h < k\leq N\\ \text{gcd}(h,k)=1 \\ \text{gcd}(k,6)=1 \\hh'\equiv -1\text{ } (k)\\ 6 |h'}} \frac{(-1)^{\frac{1}{2}(k+1)}}{k} \frac{\omega_{6h,k}\omega_{2h,k}\omega_{h,k}}{\omega_{3h,k}} e^{\frac{3\pi i h k}{2} -\frac{3\pi i h}{2k}-\frac{2\pi i nh}{k}}I_1 \left( \frac{2\pi \sqrt{2n+1  }}{3k} \right),
\end{align*}
as $N \to \infty$. Note that this term does not come up in the final formula in Theorem \ref{mainresult}. That is because, as we will see in Subsection 4.4.3, the principle part of $S_{63}$ is exactly the same as the one of $-S_{61}.$ Thus they cancel each other and none of them are part of the  final formula.  
\subsubsection{$S_{12}$} This integral also has a principle part. The calculations needed to determine its contribution as $N\to \infty$ are analogous to the ones for $S_{62}$. Note that here we need the results for the integral $\mathcal{J}$ instead of the ones for $\mathcal{I}$ from Lemma \ref{boundIntJ}. The contribution is exactly the fourth (and last) term in the final result. 
\subsubsection{$S_{13}$} 
Once again, as with $S_{23},$ the integral $S_{13}$ also has a principle and a non-principle part.  The principle part is given by
\begin{align*}
    \mathbb{S}_{13} &= \frac{1}{4} \sum_{\substack{0 \leq h < k\leq N\\ \text{gcd}(h,k)=1 \\ \text{gcd}(k,6)=1 \\hh'\equiv -1\text{ } (k)\\6|h'}} \frac{\omega_{3h,k}\omega_{2h,k}\omega_{h,k}}{\omega_{6h,k}^3}  e^{-\frac{2\pi i nh}{k}}\int_{-\vartheta'_{h,k}}^{\vartheta''_{h,k}}   e^{\frac{ \pi}{9kz} + \frac{(2n+1) \pi z}{k}}  \text{d} \Phi
    \\ &= \frac{\pi }{6\sqrt{2n+1}}\sum_{\substack{0 \leq h < k\leq N\\ \text{gcd}(h,k)=1 \\ \text{gcd}(k,6)=1 \\hh'\equiv -1\text{ } (k)\\6|h'}} \frac{1}{k}\frac{\omega_{3h,k}\omega_{2h,k}\omega_{h,k}}{\omega_{6h,k}^3}  e^{-\frac{2\pi i nh}{k}} I_1 \left( \frac{2\pi \sqrt{2n+1 }}{3k} \right),
\end{align*}
as $N \to \infty.$ As already alluded to in Subsection 4.4.1, the contribution of $S_{63}$ to the final formula as $N \to \infty$ is exactly the negative of the contribution of $S_{61}.$ To see this, note that these principle parts are identically, up to the Kloosterman sum, with which they are multiplied. However, from our previous examination of these sums (\ref{KMEQUAL}), we know that the sums $K_k^{[111]}(n)$ and $K_k^{[131]}(n)$ only differ in their sign, i.e., their sum is $0$. So the principle parts of $S_{61}$ and $S_{63}$ cancel, once we add them up. Therefore, none of them show up in the final result.  This concludes the proof of Theorem \ref{mainresult}.  \begin{flushright} \vspace{-0.5cm}$\square$
\end{flushright}
As this proof is rather technical, we want to briefly recap the takeaways here. The main formula in Theorem \ref{mainresult} consists of four terms. The fourth term is exactly the contribution of integral $S_{12}.$ The second and third term arise in the calculations when we considered $S_{22}$ and $S_{23}.$ In the case that gcd$(k,6)=3$ we found no contributions. The remaining term is the contribution from $S_{62}.$ The cancellation of the terms that we described at the very end of the proof resulted in the contributions of $S_{61}$ and $S_{63}$ not showing up in the final formula. The same cancellation effect was also observed in  \cite{bridges2023rademachertype} and \cite{mauth2025exactformula1lowerrun}  where other weight 0 mixed mock modular forms where considered. 
\section*{Data Availability} No new data was created or analyzed in this research. Data sharing is not applicable to this article.

\section*{Compliance with Ethical Standards}
The author declares that he has no conflict of interest.

\end{document}